\documentclass[a4paper,11pt]{amsart}

\usepackage{amssymb}
\usepackage{amstext}
\usepackage{amsmath}
\usepackage{amscd}
\usepackage{amsthm}
\usepackage{amsfonts}
\usepackage{enumerate}
\usepackage{graphicx}
\usepackage{latexsym}
\usepackage[all]{xy}
\usepackage[usenames]{color}

\newtheorem*{corollary*}{Corollary}
\newtheorem{theorem}{Theorem}[section]
\newtheorem{corollary}[theorem]{Corollary}
\newtheorem{lemma}[theorem]{Lemma}
\newtheorem{proposition}[theorem]{Proposition}
\newtheorem*{proposition*}{Proposition}

\theoremstyle{definition}
\newtheorem{definition}[theorem]{Definition}
\newtheorem{remark}[theorem]{Remark}

\newtheorem*{set-up}{Set-up}

\theoremstyle{remark}

\newtheorem{notation}[theorem]{Notation}

\numberwithin{equation}{section}
\renewcommand{\mod}{\mathsf{mod}}
\newcommand{\proj}{\mathsf{proj}}

\newcommand{\End}{\operatorname{End}}
\newcommand{\Hom}{\operatorname{Hom}}
\newcommand{\add}{\operatorname{\mathsf{add}}}
\newcommand{\Ext}{\operatorname{Ext}}

\newcommand{\rad}{\operatorname{rad}}

\newcommand{\Co}{\mathbf{Co}}

\newcommand{\Db}{\mathsf{D}^{\rm b}}
\newcommand{\Kb}{\mathsf{K}^{\rm b}}

\newcommand{\type}[1]{\mathbb{#1}}
\newcommand{\dpg}{\operatorname{DPic}}

\newcommand{\Sub}{\mathsf{Sub}\hspace{.01in}}
\newcommand{\Fac}{\mathsf{Fac}\hspace{.01in}}
\newcommand{\tilt}{\operatorname{tilt}}

\newcommand{\silt}{\operatorname{silt}}
\newcommand{\tsilt}{\operatorname{2-silt}}
\newcommand{\ttilt}{\operatorname{2-tilt}}

\newcommand{\La}{\Lambda}

\newcommand{\id}{{\rm id}}
\newcommand{\op}{{\rm op}}
\newcommand{\Out}{{\rm Out}}
\newcommand{\Aut}{{\rm Aut}}
\newcommand{\Inn}{{\rm Inn}}
\newcommand{\I}{\widetilde{I}}
\newcommand{\wLa}{\widetilde{\Lambda}}
\newcommand{\RHom}{\mathbf{R}\strut\kern-.2em\operatorname{Hom}\nolimits}
\newcommand{\REnd}{\mathbf{R}\strut\kern-.2em\operatorname{End}\nolimits}

\newcommand{\Lotimes}{\mathop{{\otimes}^\mathbf{L}_\Lambda}\nolimits}
\newcommand{\Lwotimes}{\mathop{{\otimes}^\mathbf{L}_{\widetilde{\Lambda}}}\nolimits}

\newcommand{\Tor}{\operatorname{Tor}\nolimits}

\newcommand{\bmu}{{\boldsymbol \mu}}
\begin{document}
\title[Derived Picard groups of preprojective algebras
of Dynkin type]
{Derived Picard groups of preprojective algebras
of Dynkin type}

\author{Yuya Mizuno}
\address{Department of Mathematics, Faculty of Science, Shizuoka University, 
836 Ohya, Suruga-ku, Shizuoka, 422-8529, Japan}
\email{yuya.mizuno@shizuoka.ac.jp}
\begin{abstract}
In this paper, we study two-sided tilting complexes of preprojective algebras
of Dynkin type. 
We construct the most fundamental class of two-sided tilting complexes, which has a group structure by derived tensor products and induces a group of auto-equivalences of the derived category.
We show that the group structure of the two-sided tilting complexes is isomorphic to the braid group of the corresponding folded graph. 
Moreover we show that these two-sided tilting complexes induce tilting mutation and any tilting complex is given as the derived tensor products of them. 
Using these results, we determine the derived Picard group of preprojective algebras 
for type $A$ and $D$.

\end{abstract}
\maketitle
\tableofcontents

\section{Introduction}

One of the most fundamental connection between 
the quiver representation theory and the root system is the discovery by Gabriel \cite{Ga}. 
He showed that indecomposable modules of the path algebra of a (simply-laced) Dynkin quiver $Q$ correspond to the positive roots of $\Delta$, where $\Delta$ is the underlying graph of $Q$.
Recently, it has turned out that the preprojective algebra allows us to give a stronger and more direct connection.  
Namely, the preprojective $\La_\Delta$, which unifies the path algebras of all quivers with the underlying graph $\Delta$, gives a representation-theoretical interpretation of the Weyl group of $\Delta$ \cite{IR,BIRS,M}. 
This fact leads to the extensive study of connections between representation theory of algebras and combinatorics, for example \cite{AM,AIRT,BIRS,GLS,IR,IRRT,IRTT,L,ORT}. 
In this paper, we investigate a further connection from the viewpoint of tilting theory and derived categories.

Recall that in \cite{AM} we classify all tilting complexes by relating them with the elements of braid group of the corresponding folded graph. The main results of \cite{AM} can be summarized as follows.

\begin{theorem}{\cite{AM}}\textnormal{ (see Theorem \ref{tilt-bij})}\label{main1}
Let $\Delta$ be a Dynkin graph, $\Delta^{\rm f}$ the folded graph of $\Delta$ (Definition \ref{def fold}) and $\La$ the preprojective algebra of $\Delta$. 

\begin{itemize}
\item[(a)] Let $B_{\Delta^{\rm f}}$ be the braid group of $\Delta^{\rm f}$ and $\tilt\La$ the set of isomorphism classes of basic tilting complexes of $\La$. Then we have a bijection 
$$B_{\Delta^{\rm f}}\longrightarrow\tilt\La,$$ 
$$\ \ \ \ \ \ \ \ \ \ \ a:=a_{i_1}^{\epsilon_{i_1}}\cdots a_{i_k}^{\epsilon_{i_k}}\mapsto
\bmu_{a}:=\bmu_{i_1}^{\epsilon_{i_1}}\circ\cdots \circ \bmu_{i_k}^{\epsilon_{i_k}}(\La),$$
where $\bmu_{}$ denotes by the irreducible left or right tilting mutation (see subsection \ref{some results} for the detail).

\item[(b)] Any basic tilting complex $T$ of $\Db(\mod\La)$ satisfies $\End_{\Kb(\proj\La)}(T)\cong\La$. In particular, the derived equivalence class coincides with the Morita equivalence class.
\end{itemize}
\end{theorem}

Thus, the result allows us to give an explicit description of tilting complexes by mutation. 
Moreover this result implies that the set of tilting complexes gives a group structure in terms of mutation. 
However, this description of tilting complexes does not allow the derived tensor product $\Lotimes$ since we do not know the bimodule structure. 
To define the natural multiplication by $\Lotimes$, 
we will consider the notion of \emph{two-sided tilting complexes} \cite{Ric2} (Definition \ref{def two-sided}). 
Recall that a two-sided tilting complex is given by the complex $X\in\Db(\mod \La^{\rm e})$ such that 
$-\Lotimes X:\Db(\mod\La)\to\Db(\mod\La)$ gives an equivalence and 
they are closed under multiplications of $\Lotimes$.
One of the main results in this paper is to give a construction of a fundamental class of two-sided tilting complexes. 
To explain our results, we give the following set-up. 

Let $\widetilde{\Delta}$ be an extended Dynkin graph obtained from $\Delta$ and $\wLa$ the preprojective algebra of $\widetilde{\Delta}$. 
Let $\I_i:=\wLa(1-e_i)\wLa$, where $e_i$ is the primitive idempotent of $\wLa$ associated with a vertex $i\in {\Delta}_0$.
Then we can define $\I_w =\I_{i_1}\I_{i_2}\cdots \I_{i_k}$ for any reduced 
expression $w=s_{i_1}\cdots s_{i_k}$ of the Weyl group $W_\Delta$ (see subsection \ref{some results}), which turns out to be a tilting $\wLa$-module by \cite{IR,BIRS}. Using this terminology, we can give a two-sided tilting complex of $\La$ as follows.

\begin{theorem}\label{main2}\textnormal{(Proposition \ref{two-sided})} 
Let $\Delta^{\rm f}$ be the folded graph of $\Delta$.  
For any $i\in \Delta_0^{\rm f}$, define a reduced 
expression $t_i\in W_\Delta$ as follows

\[
t_i:=\left\{\begin{array}{ll}
\ s_i & \mbox{if  $i=\iota(i)$ in $\Delta$},\\
\ s_is_{\iota(i)}s_i & \mbox{if there is an edge $i\stackrel{ }{\mbox{---}}\iota(i)$ in $\Delta$},\\
\ s_is_{\iota(i)} & \mbox{if no edge between $i$ and $\iota(i)$ in $\Delta$},\\
\end{array}\right.\] 
where $\iota$ is the Nakayama permutation of $\Lambda$ (subsection \ref{relation}). 
Then
$$T_i:=\sigma^{\geq -1}(\La\Lwotimes\I_{t_i}\Lwotimes\La)\ \in\Db(\mod \La^{\rm e})$$  
is a two-sided tilting complex.
\end{theorem}

Moreover we can show that the set $\{T_i\ |\ i\in \Delta_0^{\rm f}\}$ is the \emph{most fundamental} two-sided tilting complexes from the viewpoint of the theorem below.
We denote by $\langle {T}_i\ |\ i\in \Delta_0^{\rm f} \rangle$ the set of 
two-sided tilting complexes of $\La^{\rm e}$ which can be written as 
$$T_{i_1}^{\epsilon_{i_1}}\Lotimes T_{i_2}^{\epsilon_{i_2}}\Lotimes\cdots \Lotimes T_{i_k}^{\epsilon_{i_k}}\in\Db(\mod\La^{\rm e}),$$
where $i_1,\ldots, i_k\in \Delta_0^{\rm f}$ and $\epsilon_{i_j}\in\{\pm 1\}$.

Then we will obtain the following result.

\begin{theorem}\label{main3}\textnormal{(Theorem \ref{group surj})}
There is a group isomorphism
$$B_{\Delta^{\rm f}}\to\langle {T}_i\ |\ i\in \Delta_0^{\rm f} \rangle,$$
$$\ \ \ \ \ \ \ \ \ a:=a_{i_1}^{\epsilon_{i_1}}\cdots a_{i_k}^{\epsilon_{i_k}}\mapsto T_a:=T_{i_1}^{\epsilon_{i_1}}\Lotimes T_{i_2}^{\epsilon_{i_2}}\Lotimes\cdots \Lotimes T_{i_k}^{\epsilon_{i_k}},$$
and we have an isomorphism in $\Db(\mod \La)$
$$T_a\cong\bmu_{a}(\La).$$ 
In particular, the restriction of $\langle {T}_i\ |\ i\in \Delta_0^{\rm f} \rangle$ to $\Db(\mod \La)$ gives a bijection between $B_{\Delta^{\rm f}}$ and $\tilt\La$ from Theorem \ref{main1}.
\end{theorem}

Thus Theorem \ref{main3} establishes a direct connection between the braid group, two-sided tilting complexes and mutation. We remark that a similar categorical construction for Weyl groups has been studied, which also plays an important role in this paper (see subsection \ref{some results}).

Let us remark that the braid group action on a triangulated category has been widely studied  and the notion of spherical objects provides one of the important sources of it (for example \cite{ST,KS,RZ,Gr,GM}). 
However, there is no spherical object in $\La$ and our method can be applied in this general situation. 
We also remark that Rickard and Keller gave general methods to construct a two-sided tilting complex from a given (one-sided) tilting complex \cite{Ric2,Ke1,Ke2}. 
Nonetheless, these constructions are not very explicit so that it is not clear that these complexes satisfy the braid relations in our setting.

Moreover the family $\langle {T}_i\ |\ i\in \Delta_0^{\rm f} \rangle$ directly induces a group of auto-equivaleces of $\Db(\mod \La)$ and 
it provides a crucial step to understand the \emph{derived Picard group} (Definition \ref{def dpg}). 
The notion of the derived Picard group was introduced by Rouquier-Zimmermann \cite{RZ} and Yekutieli \cite{Y}, which is a group of auto-equivalences consisting of standard functors modulo functorial isomorphisms. This notion provides an important invariant of the derived category and it is also closely related to the Hochschild cohomology \cite{Ke3}.  
Using Theorem \ref{main3}, we determine the derived Picard group as follows.

\begin{theorem}\label{main4}\textnormal{(Theorem \ref{dpg thm})}
Let $\La$ be a preprojective algebra of type $\type{A}_n$ or $\type{D}_n$. 
There is a group isomorphism 
$$\Theta:\Out(\La)\ltimes B_{\Delta^{\rm f}}\to\dpg(\Lambda),\ (\phi,a)\mapsto {}_\phi\La\Lotimes{T}_{a}.$$
\end{theorem}

\textbf{Notation}
Throughout this paper, 
let $K$ be an algebraically closed field and $D:=\Hom_K(-,K)$.
For an algebra $\La$ over $K$, we denote by $\mod\Lambda$ the category of finitely generated right $\Lambda$-modules and by $\proj\La$ the category of finitely generated projective $\La$-modules. 
We denote by $\Db(\mod\La)$ the bounded derived category of $\mod\La$ and by $\Kb(\proj\Lambda)$ the bounded homotopy category of $\proj\La$. 
Let $\La^{\rm e}:=\La^\op\otimes_K\La$, where $\La^\op$ denote the opposite algebra of $\La$, and we assume that $K$ acts centrally and identify $\La^{\rm e}$-modules with $\La$-bimodules.\\


\section{Preliminaries} 

In this section, we recall some definitions and results, which are necessary in this paper.

\subsection{Preprojective algebras}\label{nakayama per}

Let $\Delta$ be a simply-laced (i.e. type $\type{A},\type{D},\type{E}$) Dynkin graph and we denote by $\Delta_0$ the vertices of $\Delta$. Let $\La=\La_\Delta$ be the preprojective algebra of $\Delta$ (see \cite{GP,DR,Rin,BGL} for the background). It is finite dimensional and selfinjective \cite[Theorem 4.8]{BBK}. 
Without loss of generality, we may suppose that vertices are given as Figure \ref{figure} (these choices make the notation simpler) and  
let $e_i$ be the primitive idempotent of $\La$ associated with $i\in \Delta_0$. 
We denote the Nakayama permutation of $\La$ by $\iota:\Delta_0\to \Delta_0$ (i.e. $D(\La e_{\iota(i)})\cong e_{i}\La$). 
Then, we have $\iota=\id$ if  $\Delta$ is type $\type{D}_{2n},\type{E}_7$ and $\type{E}_8$. Otherwise, we have $\iota^2=\id$ and it is given as follows.

\[\left\{\begin{array}{ll}
\iota(1)=1\ \mbox{and}\ \iota(i)=i+n-1\ \mbox{for}\ i\in\{2,\cdots, n\}\ \ \ \ \ &\mbox{if $\type{A}_{2n-1}$}\\
\iota(i)=i+n\ \mbox{for}\ i\in\{1,\cdots, n\} \ \ \ &\mbox{if $\type{A}_{2n}$}\\
\iota(1)=2n+1\ \mbox{and}\ \iota(i)=i\ \mbox{for}\ i\notin\{1,2n+1\} \ \ \ &\mbox{if $\type{D}_{2n+1}$}\\
\iota(3)=5, \iota(4)=6\ \mbox{and}\ \iota(i)=i \ \mbox{for}\ i\in\{1,2\}&\mbox{if $\type{E}_{6}$.}\\
\end{array}\right.\]

\begin{figure}
\[\def\arraystretch{2}
\begin{array}{ll}
\type{A}_{2n-1}\ : &
\begin{array}{c}
\xymatrix@C5pt@R10pt{
n \ar@{-}[r] & \cdots \ar@{-}[r] & 2 \ar@{-}[r] &\ar@{-}[r]  1 &\ar@{-}[r](n+1)& \cdots \ar@{-}[r] &   (2n-1). 
}
\end{array} \\
\type{A}_{2n}\ : &
\begin{array}{c}
\xymatrix@C5pt@R10pt{
n \ar@{-}[r]  & \cdots \ar@{-}[r] & 2 \ar@{-}[r] &\ar@{-}[r]  1 &\ar@{-}[r](n+1)& \cdots \ar@{-}[r]   &   2n. 
}
\end{array} \\

\type{B}_n\ (n\geq1): &
\begin{array}{c}
\xymatrix{
1 \ar@{-}[r]^{4} & 2 \ar@{-}[r]  & \cdots &\ar@{-}[r]  & n -1\ar@{-}[r]  & n. 
}
\end{array} \\
\type{D}_n\ (n\geq4): &
\begin{array}{c}
\xymatrix@R=0.3cm{
1  \ar@{-}[rd]&   &   &        & \\
  & 2 \ar@{-}[r]  &3 \ar@{-}[r]  & \cdots \ar@{-}[r]  & n-1.  \\
n \ar@{-}[ru] &   &   &        &
}
\end{array} \\
\type{E}_n\ (n=6,7,8): &
\begin{array}{c}
\xymatrix{
  & & 1  \ar@{-}[d]&   &   &        & \\
4 \ar@{-}[r]  & 3 \ar@{-}[r]  &
2 \ar@{-}[r]   & 5 \ar@{-}[r] & \cdots \ar@{-}[r]& n. 
}
\end{array} \\
\type{F}_4\ : &
\begin{array}{c}
\xymatrix{
1 \ar@{-}[r] & 2 \ar@{-}[r]^4 & 3 \ar@{-}[r]  & 4. 
}
\end{array} \\

\end{array}\]
\caption{}\label{figure}\end{figure}

\subsection{Weyl group}\label{relation}

Let $\Delta$ be a graph given as Figure \ref{figure}. 
The \emph{Weyl group} $W_\Delta$ associated with $\Delta$ is defined by the generators 
$s_i$ and relations 
$(s_is_j)^{m_\Delta(i,j)}=1$, 
where  

\[
m_\Delta(i,j):=\left\{\begin{array}{ll}
1\ \ \ \ \ &\mbox{if $i=j$,}\\
2\ \ \ \ \ &\mbox{if no edge between $i$ and $j$ in $\Delta$,}\\
3\ \ \ \ \ &\mbox{if there is an edge $i\stackrel{ }{\mbox{---}}j$ in $\Delta$,}\\
4\ \ \ \ &\mbox{if  there is an edge $i\stackrel{4}{\mbox{---}}j$ in $\Delta$.}\\
\end{array}\right.\]

For $w\in W_\Delta$, we denote by $\ell(w)$ the length of $w$.
 
Let $\Delta$ be a simply-laced Dynkin graph, $\La$ the preprojective algebra and $\iota$ the Nakayama permutation of $\Lambda$. 
Then $\iota$ acts on an element of the Weyl group $W_\Delta$ by $\iota(w):=s_{\iota(i_1)}s_{\iota(i_2)}\cdots s_{\iota(i_k)}$ 
for $w=s_{i_1}s_{i_2}\cdots s_{i_k}\in W_\Delta$. 
We define the subgroup $W^\iota_\Delta$ of $W_\Delta$ by 
\[W^\iota_\Delta:=\{w\in W_\Delta\ |\ \iota(w)=w \}.\]

Note that we have $w_0ww_0=\iota(w)$ for $w\in W_\Delta$ for the longest element $w_0$ of $W_\Delta$.

Moreover we have the following result (see \cite[Chapter 13]{C},\cite[Theorem 3.1]{AM}).

\begin{theorem}\label{folding}
Let $\Delta$ be a simply-laced Dynkin graph whose vertices are given as Figure \ref{figure} and $W_\Delta$ the Weyl group of $\Delta$. 
Let $\Delta^{\rm f}$ be a graph given by the following type. 
\[\begin{array}{|c|c|c|c|c|c|c|c|}
\hline                                                                       
\Delta & \type{A}_{2n-1},\type{A}_{2n}&\type{D}_{2n} &\type{D}_{2n+1} & \type{E}_6  & \type{E}_7 & \type{E}_8 \\ \hline
\Delta^{\rm f} &\type{B}_n  &\type{D}_{2n}        &\type{B}_{2n}      & \type{F}_4 & \type{E}_7  & \type{E}_8 \\ \hline
\end{array}\]

Then we have $W_\Delta^\iota=\langle t_i\ |\ i \in \Delta^{\rm f}_0\rangle$, where 
\[\tag{T}\label{fold}
t_i:=\left\{\begin{array}{ll}
\ s_i & \mbox{if  $i=\iota(i)$ in $\Delta$},\\
\ s_is_{\iota(i)}s_i & \mbox{if there is an edge $i\stackrel{ }{\mbox{---}}\iota(i)$ in $\Delta$},\\
\ s_is_{\iota(i)} & \mbox{if no edge between $i$ and $\iota(i)$ in $\Delta$}.\\
\end{array}\right.\]
and $W_\Delta^\iota$ is isomorphic to $W_{\Delta^{\rm f}}$. 
\end{theorem}

For the convenience, we introduce the following terminology.

\begin{definition}\label{def fold}
We call the graph $\Delta^{\rm f}$ given in Theorem \ref{folding} the \emph{folded graph} of $\Delta$.
\end{definition}

Moreover we denote the braid group by $B_{\Delta^{\rm f}}$, which is defined by generators $a_i$ $(i\in \Delta^{\rm f}_0)$ with relations $(a_ia_j)^{m_{\Delta^{\rm f}}(i,j)}=1$ for $i\neq j$.

\subsection{Silting and tilting complexes}

In this subsection, we recall the notion of tilting and silting complexes. See \cite{Ric1,AI} for additional background.

\begin{definition}\label{def silt obj}
We call a complex $P$ in $\Kb(\proj\Lambda)$ \emph{silting} (respectively, \emph{tilting})
if it satisfies $\Hom_{\Kb(\proj\Lambda)}(P, P[i])=0$ for any $i>0$ (respectively, $i\neq0$) and the smallest thick subcategory containing $P$ is $\Kb(\proj\Lambda)$.
We denote by $\silt\La$ (respectively, $\tilt\La$) the set of isomorphism classes of basic silting complexes (respectively, tilting complexes) in $\Kb(\proj\Lambda)$. 
Moreover, let $\tsilt\Lambda$ (respectively, $\ttilt\Lambda$) be the subset of $\silt\Lambda$ (respectively, $\tilt\Lambda$) consisting of two-term (i.e. it is concentrated in the degree 0 and $-1$) complexes.
\end{definition}

Moreover we recall mutation for silting complexes.

\begin{definition}\label{defsm}
Let $P$ be a basic silting complex of $\Kb(\proj\Lambda)$ and decompose it as $P=X\oplus M$. 
We take a triangle
\[\xymatrix{
X \ar[r]^f & M' \ar[r] & Y \ar[r] & X[1]
}\]
with a minimal left ($\add M$)-approximation $f$ of $X$.
Then $\mu_X^+(P):=Y\oplus M$ is again a silting complex, 
and we call it the \emph{left mutation} of $P$ with respect to $X$.
Dually, we define the right mutation $\mu_X^-(P)$. 
Mutation means either left or right mutation.  
If $X$ is indecomposable, then we say that 
mutation is \emph{irreducible}.

Moreover, if $P$ and $\mu_X^+(P)$ are tilting complexes, then we call it the (left) \emph{tilting mutation}. In this case, if there exists no non-trivial direct summand $X'$ of $X$ such that $\mu_{X'}^+(T)$ is tilting, then we say that tilting mutation is \emph{irreducible}. 
\end{definition}

\subsection{Summary of previous results}\label{some results}

In this subsection, we review some known results.
Let $\Delta$ be a simply-laced Dynkin graph with $\Delta_0:=\{1,\ldots, n\}$, $\Delta^{\rm f}$ the folded graph of $\Delta$ and $\La$ the preprojective algebra of $\Delta$.
Let $I_i:=\La(1-e_i)\La$, where $e_i$ the primitive idempotent of $\La$ associated with $i\in \Delta_0$. 
We denote by $\langle I_1,\ldots,I_n\rangle$ the set of ideals of $\La$ which can be written as $I_{i_1}I_{i_2}\cdots I_{i_k}$ for some $k\geq0$ and $i_1,\ldots,i_k\in \Delta_0$. 
Then we have the following result \cite{BIRS,IR,AM}.

\begin{theorem}\label{tau-weyl3}
We have  a bijection $W_\Delta\to\langle I_1,\ldots,I_n\rangle$, which is given by $w\mapsto I_w =I_{i_1}I_{i_2}\cdots I_{i_k}$ for any reduced 
expression $w=s_{i_1}\cdots s_{i_k}$.
\end{theorem}

\begin{proof}
See \cite[Theorem 2.14]{M}.
\end{proof}

Next, for $i\in\Delta^{\rm f}_0$, we define $\bmu_i^+(\La)$ in $\Kb(\proj\Lambda)$, where $\bmu_i^+$ is given as a composition of left silting mutation as follows 
\[\bmu_i^+:=\left\{\begin{array}{ll}
\ \mu_{i}^+ & \mbox{if  $i=\iota(i)$ in $\Delta$},\\
\ \mu_{i}^+\circ\mu_{{\iota(i)}}^+\circ\mu_{i}^+& \mbox{if there is an edge $i\stackrel{ }{\mbox{--}}\iota(i)$ in $\Delta$},\\
\ \mu_{i}^+\circ\mu_{{\iota(i)}}^+ & \mbox{if no edge between $i$ and $\iota(i)$ in $\Delta$}.\\ \end{array}\right.\]

On the other hand, for $i\in\Delta^{\rm f}_0$, we let 
\[e_{i}^\iota:=\left\{\begin{array}{ll}
 e_i & \mbox{if  $i=\iota(i)$ in $\Delta$},\\
e_i+e_{\iota(i)} & \mbox{if  $i\neq\iota(i)$ in $\Delta$}.
\end{array}\right.\]

It is easy to check that
$\bmu_i^+(\La)=\mu_{(e_{i}^\iota\La)}^+(\La)$ and hence we have a two-term tilting complex 
$$\bmu_i^+(\La)=\left\{\begin{array}{cccc}
\stackrel{-1}{ e_{i}^\iota\La}&\stackrel{f}{\longrightarrow}&
\stackrel{0}{R}\\
&\oplus&&\in\Kb(\proj\La)\\
&&(1-e_{i}^\iota)\La
\end{array}\right.$$ 
where $f$ is a minimal left $(\add((1-e_{i}^\iota)\La))$-approximation. 

Then $\bmu_i^+$ gives an irreducible left tilting mutation of $\La$ and any  irreducible left tilting mutation of $\La$ is given as $\bmu_i^+$ for some $i\in \Delta_0^{\rm f}$ \cite[Theorem 4.2]{AM}. 
Dually, we define $\bmu_i^-$. Note that $\bmu_i^+\circ\bmu_i^-=\id$ (\cite[Proposition 2.33]{AI}).

Then these results \cite{AIR,M,AM} are fundamental.

\begin{theorem}\label{tau-weyl2}
\begin{itemize}
\item[(a)] We have a bijection 
$$W_\Delta \longrightarrow \tsilt\La,\ s_{i_1}\cdots s_{i_k}\mapsto
\mu_{i_1}^+\circ\cdots \circ \mu_{i_k}^+(\La),$$
where $s_{i_1}\cdots s_{i_k}$ is a reduced 
expression.
\item[(b)] 
We have a bijection 
$$W_{\Delta^{\rm f}} \longrightarrow \ttilt\Lambda,\ s_{i_1}\cdots s_{i_k}\mapsto
\bmu_{i_1}^+\circ\cdots \circ \bmu_{i_k}^+(\La),$$
where $s_{i_1}\cdots s_{i_k}$ is a reduced 
expression.
\end{itemize}
\end{theorem}

\begin{proof}
See \cite[Theorem 4.1,4.2]{AM}.
\end{proof}

Moreover we recall the main result of \cite{AM}. 
Let $B_{\Delta^{\rm f}}$ be the braid group generated by $a_i$ $(i\in \Delta_0^{\rm f})$. 

Then we have the following result.
\begin{theorem}\label{tilt-bij}
\begin{itemize}
\item[(a)] We have a bijection 
$$B_{\Delta^{\rm f}}\longrightarrow\tilt\La,$$ 
$$\ \ \ \ \ \ \ \ \ \ \ a=a_{i_1}^{\epsilon_{i_1}}\cdots a_{i_k}^{\epsilon_{i_k}}\mapsto
\bmu_a(\La):=\bmu_{i_1}^{\epsilon_{i_1}}\circ\cdots \circ \bmu_{i_k}^{\epsilon_{i_k}}(\La).$$
\item[(b)]Any basic tilting complex $T$ of $\Kb(\proj\La)$ satisfies $\End_{\Kb(\proj\La)}(T)\cong\La$. 
\end{itemize}
\end{theorem}


\section{Two-sided tilting complexes}\label{sec two-sided}
In this section, we will study two-sided tilting complexes of the preprojective algebra of Dynkin type.
We will construct a fundamental class of two-sided tilting complexes and show that 
they induce irreducible tilting mutation. 
This fact allows us to show that any tilting complexes are obtained as a composition of derived tensor products of these two-sided tilting complexes.

First we give the following set-up.

\begin{notation}\label{main notation}
Let $\Delta$ be a simply-laced Dynkin graph, $\La$ the preprojective algebra of $\Delta$. 
Let $\Delta^{\rm f}$ be the folded graph of $\Delta$. 
Let $\widetilde{\Delta}$ be an extended Dynkin graph obtained from $\Delta$ by adding a vertex $0$ (i.e. $\widetilde{\Delta}_0=\{0\}\cup \Delta_0$) with the associated edges. 
We denote by $\wLa$ the $\mathfrak{m}$-adic completion of the preprojective algebra of $\widetilde{\Delta}$, where $\mathfrak{m}$ is the ideal generated by all arrows. It implies that 
the Krull-Schmidt theorem holds for finitely generated projective $\wLa$-modules.
Moreover we denote by $\I_i:=\wLa(1-e_i)\wLa$, where $e_i$ is the primitive idempotent of $\wLa$ associated with $i\in \widetilde{\Delta}_0$. 
Then for $w\in W_\Delta$, we can define $\I_w$ as Theorem \ref{tau-weyl3}, which is a tilting $\wLa$-module \cite{IR,BIRS}.

Note that, since we have the natural surjection $\wLa\to\La$, 
we have the restriction functor $\Db(\mod\La^{\rm e})\to\Db(\mod(\wLa^{\op}\otimes_K\La))$ and hence $X\in \Db(\mod\La^{\rm e})$ can be regarded as a complex in $\Db(\mod(\wLa^{\op}\otimes_K\La))$. 

Let $\Aut(\La)$ be the group of automorphisms of $\La$. 
For a $\La^{\rm e}$-module $X$ and $\phi,\psi\in\Aut(\La)$, we denote by ${}_\psi X_\phi$ the $\La^{\rm e}$-module whose right action is given by $x\cdot\lambda:=
x\phi(\lambda)$ and left action is given by $\lambda'\cdot x:=\psi(\lambda')x$ for $x\in X$ and $\lambda,\lambda'\in\La$. 
Let $\nu:=D\Hom_\La(-,\La)$ be the Nakayama functor. 
By abuse of notation, we also denote the Nakayama automorphism by $\nu$ so that 
$\nu(\La) \cong {}_1(\La)_\nu\cong  {}_1(\La)_{\nu^{-1}}$ (see \cite[IV.Proposition 3.15]{SY} and \cite[Theorem 4.8]{BBK}). 
\end{notation}

Then we give the following lemma. 

\begin{lemma}\label{homology}
For $w\in W_{\Delta}$, 
we have isomorphisms in $\Db(\mod(\wLa^{\op}\otimes_K\La))$
$$H^0(\I_{w}\Lwotimes\La)\cong I_w,\ H^{-1}(\I_{w}\Lwotimes\La)\cong{}_1(\La/I_w)_\nu\ \textnormal{and}\ H^{j}(\I_{w}\Lwotimes\La)\cong0$$ 
for any $j\neq 0,-1$.  
\end{lemma}

\begin{proof}
From the definition, we have $H^0(\I_{w}\Lwotimes\La)=\I_w\otimes_{\wLa}\La\cong I_w.$

Moreover, we have 
\begin{eqnarray*}
H^{-1}(\I_{w}\Lwotimes\La)&=&\Tor_1^{\wLa}(\I_{w},\La)\\
&\cong&D\Ext^1_{\wLa}(\I_{w},D\La)\\
&\cong&D\Ext^2_{\wLa}(\wLa/\I_{w},D\La)\\
&\cong&\Hom_{\wLa}(D\La,\wLa/\I_{w})\ \ \ \ \ \ (\textnormal{2-CY duality})\\
&\cong&\Hom_{\La}(D\La,\La/I_{w})\\
&\cong&{}_1(\La/I_w)_\nu.\\
\end{eqnarray*}

Since $\I_{w}$ is a tilting module and hence the projective dimension is at most one, we have $H^{j}(\I_{w}\Lwotimes\La)\cong0$ 
for any $j\neq 0,-1$.  
\end{proof}

For $w\in W_\Delta$, we denote by 
$$B_w:=\La\Lwotimes\I_{w}\Lwotimes\La\in\Db(\mod\La^{\rm e}).$$ 
Then we give the following lemma.

\begin{lemma}\label{homology2}
For $w\in W_{\Delta}$, we have isomorphisms in $\Db(\mod(\wLa^{\op}\otimes_K\La))$
$$H^0(B_w)\cong I_w,\ H^{-1}(B_w)\cong{}_1(\La/I_w)_\nu,\ H^{-2}(B_w)\cong{}_1(I_w)_\nu,\ H^{-3}(B_w)\cong\La/I_w$$ 
and $H^{j}(B_w)\cong0$ 
for any $j\neq 0,-1,-2,-3$.  
\end{lemma}

\begin{proof}
We write $X:=\I_{w}\Lwotimes\La$, 
$X^{-1}:=H^{-1}(X)$ and  $X^{0}:=H^{0}(X)$ for simplicity.

Take the canonical triangle 

$$\xymatrix@C30pt@R30pt{\cdots\ar[r]&\sigma^{\leq-1}(X)\ar[r]^{}& X\ar[r]^{ }&\sigma^{\geq0}(X)\ar[r]& (\sigma^{\leq-1}X)[1]\ar[r]&\cdots,}$$ 
where $\sigma$ denotes by the truncation functor.

By Lemma \ref{homology}, we can write it as
$$\xymatrix@C30pt@R30pt{\cdots\ar[r]&X^{-1}[1]\ar[r]^{}& X\ar[r]^{ }&X^0\ar[r]& X^{-1}[2]\ar[r]&\cdots.}$$ 
Then, applying the functor $\La\Lwotimes-$ to the triangle, we have the following triangle 
$$\xymatrix@C30pt@R30pt{\cdots\ar[r]&\La\Lwotimes X^{-1}[1]\ar[r]^{}&\La\Lwotimes X\ar[r]^{ }&\La\Lwotimes X^0\ar[r]&\La\Lwotimes  X^{-1}[2]\ar[r]&\cdots.}$$

Taking the homology, we have the following long exact sequence 

\[
\xymatrix@C15pt@R5pt
{ 0 \ar[r]^{} &  H^{-3}(\La\Lwotimes X^{-1}[1])\ar[r] &H^{-3}(\La\Lwotimes X)   \ar[r] &H^{-3}(\La\Lwotimes X^0)\ar[r] &\\
 \ar[r]^{} &  H^{-2}(\La\Lwotimes X^{-1}[1])\ar[r] &H^{-2}(\La\Lwotimes X)   \ar[r] &H^{-2}(\La\Lwotimes X^0)\ar[r] &\\
\ar[r]^{f\ \ \ \ \ \ \ \ \ \ \ \ } &  H^{-1}(\La\Lwotimes X^{-1}[1])\ar[r] &H^{-1}(\La\Lwotimes X)   \ar[r] &H^{-1}(\La\Lwotimes X^0)\ar[r] &\\
\ar[r]  &    H^0(\La\Lwotimes X^{-1}[1]) \ar[r] & H^0(\La\Lwotimes X)\ar[r]  &H^0(\La\Lwotimes X^0) \ar[r]  &  0. }
\]

Then we have 

\begin{eqnarray*}
H^{-3}(\La\Lwotimes X^{-1}[1])&=&H^{-2}(\La\Lwotimes X^{-1})\\
&\cong&\Tor_2^{\wLa}(\La,X^{-1})\\
&\cong&D\Ext^2_{\wLa}(\La,D(X^{-1}))\\
&\cong&\Hom_{\wLa}(D(X^{-1}),\La) \\
&\cong&\Hom_{\La}(D(X^{-1}),\La)\\
&\cong&{}_1(X^{-1})_\nu.
\end{eqnarray*}

Similarly, we have 
$H^{-2}(\La\Lwotimes X^{0})\cong{}_1(X^0)_\nu.$ 

On the other hand, we have 
\begin{eqnarray*}
H^{-2}(\La\Lwotimes X^{-1}[1])&=&H^{-1}(\La\Lwotimes X^{-1})\\
&\cong&\Tor_1^{\wLa}(\La,X^{-1})\\
&\cong&D\Ext^1_{\wLa}(\La,D(X^{-1}))\\
&\cong&D\Ext^1_{\La}(\La,D(X^{-1}))\\
&\cong&0.
\end{eqnarray*}
Similarly, we have $H^{-1}(\La\Lwotimes X^0)\cong0$. 
Moreover, since $\mathrm{gl.dim}\wLa\leq 2$ (\cite[Proposition II.1.3]{BIRS}), we get $H^{-3}(\La\Lwotimes X^0)\cong0$ and $H^0(\La\Lwotimes X^{-1}[1])\cong0$.  

Thus Lemma \ref{homology} implies 
$H^{-2}(\La\Lwotimes X^{0})\cong{}_1(I_w)_\nu$ and $H^{-1}(\La\Lwotimes X^{-1}[1])\cong H^{0}(\La\Lwotimes X^{-1})\cong {}_1(\La/I_w)_\nu$. 
Since $(\Fac I_w,\Sub(\La/I_w))$ is a torsion pair \cite[Proposition 4.2]{M}, 
we have $f=0$. 
Consequently, we have $H^0(\La\Lwotimes X)\cong H^{0}(\La\Lwotimes X^{0})\cong I_w$, 
$H^{-1}(\La\Lwotimes X)\cong {}_1(\La/I_w)_\nu$, $H^{-2}(\La\Lwotimes X)\cong {}_1(I_w)_\nu$ and  $H^{-3}(\La\Lwotimes X)\cong \La/I_w$. 
Therefore Lemma \ref{homology} shows the assertion. 
\end{proof}

Then by Lemma \ref{homology} and \ref{homology2}, we obtain the following result.

\begin{proposition}\label{quasi-iso}
For $w\in W_{\Delta}$, we have an isomorphism in $\Db(\mod(\wLa^{\op}\otimes_K\La))$
$$\sigma^{\geq -1}(B_w)\cong\I_{w}\Lwotimes\La.$$
\end{proposition}

\begin{proof}
Let $Y:=\langle e_0\rangle=\wLa e_0\wLa$. 
Take a short exact sequence 
$\xymatrix@C30pt@R30pt{0\ar[r]&Y\ar[r]^{}& \wLa \ar[r]& \La\ar[r]&0.}$

We let $X:=\I_{w}\Lwotimes\La$. 
Then, applying the functor $-\Lwotimes X$ to the exact sequence, we have the triangle 
\[\tag{1}\label{tri seq}\xymatrix@C30pt@R30pt{\cdots\ar[r]&Y\Lwotimes X\ar[r]^{}&\wLa\Lwotimes X\ar[r]^{}&\La\Lwotimes X\ar[r]&Y\Lwotimes X[1]\ar[r]&\cdots.}\]  
Then we will show the composition of the morphisms 
$$\varphi:\xymatrix@C30pt@R10pt{\wLa\Lwotimes X(\cong X)\ar[r]^{  }&\La\Lwotimes X\ar[r]^{\sigma^{\geq -1}\ \ \ \ }&\sigma^{\geq -1} (\La\Lwotimes X)}$$
is an isomorphism in $\Db(\mod(\wLa^{\op}\otimes_K\La))$.

By taking the homology of the sequence (\ref{tri seq}), we have the following long exact sequence

\[
\xymatrix@C15pt@R5pt
{ \cdots \ar[r] &  H^{-1}(Y\Lwotimes X )\ar[r] &H^{-1}(X)   \ar[r]^{g_{-1}\ \ } &H^{-1}(\La\Lwotimes X)\ar[r] &\\
\ar[r]  &    H^0(Y\Lwotimes X) \ar[r] & H^0(X)\ar[r]^{g_{0}\ \ }  &H^0(\La\Lwotimes X) \ar[r]  &  0. }
\]

Then we have 
\begin{eqnarray*}
H^0(Y\Lwotimes X)&\cong &Y\otimes_{\wLa} \I_{w}\otimes_{\wLa} \La\\
&\cong&\langle e_0\rangle\otimes_{\wLa} \I_{w}/\langle e_0\rangle\\
&\cong&0.
\end{eqnarray*}

Then, by the above exact sequence together with Lemmas \ref{homology} and \ref{homology2},  
$g_{-1}$ and $g_{0}$ are isomorphisms. 
Thus we get the conclusion.
\end{proof}

Now we recall the following definition (we refer to \cite{Ric2} for details). 

\begin{definition}\label{def two-sided}
Let $A$ and $B$ be finite dimensional algebras. 
If a complex $T$ of $(B^\op\otimes_K A)$-modules satisfies the following equivalent conditions, 
then we call $T$ a \emph{two-sided tilting complex}.
\begin{itemize}
\item[(i)] $T$ is a tilting complex of $A$ and the left multiplication morphism $B\to\RHom_A(T,T)$ is an isomorphism in $\Db(\mod B^{\rm e})$.
\item[(ii)] $T$ is a tilting complex of $B$ and the right multiplication morphism $A\to\RHom_{B^\op}(T,T)$ is an isomorphism in $\Db(\mod A^{\rm e})$.
\item[(iii)] 
$T$ is biperfect (i.e. $T\in \Kb(\proj A)$ and $T\in \Kb(\proj B^{\op})$) and 
there exists a biperfect complex $U$ of $(A^\op\otimes_K B)$-modules such that  
$$U\otimes_B^\mathbf{L} T\cong A\ \textnormal{in}\ \Db(\mod A^{\rm e})\ \textnormal{and}\ T\otimes_A^\mathbf{L} U\cong B\ \textnormal{in}\ \Db(\mod B^{\rm e}).$$
\end{itemize}

In this case, we have $U$ 
and denote it by $T^{-1}$. 
The functor $-\otimes_B^\mathbf{L} T$ is called a \emph{standard functor} \cite{Ric2} and it gives an equivalence between $\Db(\mod B)$ and $\Db(\mod A)$.
\end{definition}

For any $i\in \Delta_0^{\rm f}$, define $t_i$ as (\ref{fold}) of Theorem \ref{folding}. 
We denote by 
$$T_i:=\sigma^{\geq -1}(B_{t_i})=\sigma^{\geq -1}(\La\Lwotimes\I_{t_i}\Lwotimes\La)\ \in\Db(\mod \La^{\rm e}).$$

\begin{lemma}\label{two-sided lemma}
For any $i\in \Delta_0^{\rm f}$, 
we have an isomorphism in $\Db(\mod\La)$ 
$$T_i\cong\bmu_i^+(\La).$$ 
\end{lemma}
\begin{proof}
By Proposition \ref{quasi-iso}, we have $T_i\cong\I_{t_i}\Lwotimes\La$ in $\Db(\mod\La)$.
On the other hand, by \cite[Proposition 5.2]{AM}, we have $\I_{t_i}\Lwotimes\La\cong\bmu_i^+(\La)$. 
\end{proof}

Then we show that $T_i$ gives a two-sided tilting complex. 

\begin{proposition}\label{two-sided}
For any $i\in \Delta_0^{\rm f}$, $T_i$ is a two-sided tilting complex. 
\end{proposition}

\begin{proof}
We show the condition (i) of Definition \ref{def two-sided}. 
From Theorem \ref{tau-weyl2} and Lemma \ref{two-sided lemma}, $T_i$ is a tilting complex of $\La$. 
Then, we will show that the left multiplication $\La\to\RHom_\La(T_i,T_i)$ is an isomorphism in $\Db(\mod \La^{\rm e})$.

We recall some results from \cite[Lemma 5.3, Proposition 5.4]{AM}. 
Let $w_0$ be the longest element of $W_\Delta$. 
Since $\I_{w_0}=\langle e_0\rangle$, we have the following exact sequence
$$\xymatrix@C30pt@R30pt{0\ar[r]&\I_{w_0}\ar[r]^{}& \wLa \ar[r]& \La\ar[r]&0.}$$ 
Then applying the functors $\I_{t_i}\Lwotimes-$ and $-\Lwotimes\I_{t_i}$ to the exact sequence, we have the 
following commutative diagram  

$$\xymatrix@C25pt@R15pt{\I_{t_i}\Lwotimes\I_{w_0}\ar[r]^{}\ar[d]_{\cong}&\I_{t_i}\Lwotimes\wLa\ar[r]^{ }\ar[d]_{\cong}&\I_{t_i}\Lwotimes\La\ar[r]\ar@{.>}[d]^{r}&\I_{t_i}\Lwotimes\I_{w_0}[1]\ar[d]_{\cong}\\
\I_{w_0}\Lwotimes\I_{t_i}\ar[r]^{}&\wLa\Lwotimes\I_{t_i}\ar[r]^{ }&\La\Lwotimes\I_{t_i}\ar[r]&\I_{w_0}\Lwotimes\I_{t_i}[1],}$$ 
 and the isomorphism $r$ by \cite[Lemma 5.3, Proposition 5.4]{AM}.

On the other hand, $\I_{t_i}$ is a two-sided tilting complex  and the left multiplication gives an isomorphism $\wLa\cong\Hom_{\wLa}(\I_{t_i},\I_{t_i})$ \cite[section II.1]{BIRS}. Then,  
we obtain  
\begin{eqnarray*}
\La&\cong&\La\Lwotimes\wLa\\ 
&\cong&\La\Lwotimes\RHom_{\wLa}(\I_{t_i},\I_{t_i})\\
&\cong&\RHom_{\wLa}(\I_{t_i},\La\Lwotimes\I_{t_i})\\
&\cong&\RHom_{\wLa}(\I_{t_i},\I_{t_i}\Lwotimes\La)\\ 
&\cong&\RHom_{\La}(\I_{t_i}\Lwotimes\La,\I_{t_i}\Lwotimes\La),
\end{eqnarray*}
and 
Proposition \ref{quasi-iso} gives an isomorphism $\RHom_{\La}(T_i,T_i)\cong\RHom_{\La}(\I_{t_i}\Lwotimes\La,\I_{t_i}\Lwotimes\La)$.  
Then we have the isomorphism $\La\to\RHom_\La(T_i,T_i)$  given by the left multiplication.  
\end{proof}

Moreover we will show that these tilting complexes satisfy braid relations. 

For this purpose, we recall the following result \cite[Proposition II.1.5, Proposition II.1.10]{BIRS}(\cite[Proposition 6.1, Theorem 6.5]{IR}).

\begin{proposition}\label{affine braid}
Let $w,v \in W_{\widetilde{\Delta}}$. If $\ell(wv)=\ell(w)+\ell(v)$, 
then we have isomorphisms in $\Db(\mod{\wLa}^{\rm e})$
$$\I_w\Lwotimes\I_v\cong\I_w\otimes_{\wLa}\I_v\cong\I_{wv}.$$
\end{proposition}

Recall that we denote by 
$B_w=\La\Lwotimes\I_{w}\Lwotimes\La\in\Db(\mod\La^{\rm e}).$
Then we give the following key proposition.

\begin{proposition}\label{braid condition1}
Let $w,v\in  W_\Delta$. If $\ell(wv)=\ell(w)+\ell(v)$, 
then we have an isomorphism in $\Db(\mod\La^{\rm e})$ 
$$\sigma^{\geq-1}(B_w)\Lotimes \sigma^{\geq-1}(B_v)\cong\sigma^{\geq-1}(B_{wv}).$$
\end{proposition}

\begin{proof}
For simplicity, we write $X^{\geq-1}:=\sigma^{\geq-1}(X)$ and $X^{\leq-2}:=\sigma^{\leq-2}(X)$ for $X\in \Db(\mod\La^{\rm e})$. 

First we have 
\begin{eqnarray*}
H^{-2}(B_w^{\leq-2}\Lotimes B_v^{\geq-1})
&=&H^{0}(B_w^{\leq-2}[-2]\Lotimes B_v^{\geq-1})\\
&\cong&H^{0}(B_w^{\leq-2}[-2])\otimes_\La H^{0}(B_v^{\geq-1})\\
&\cong&{}_1(I_w)_\nu\otimes_\La I_{v}\ \ \ \ \ \ \ \ \ \ \ \ \ \ \ \ \ \ \ \ (\mathrm{Lemma}\ \ref{homology2})\\
&\cong&{}_1(I_{wv})_\nu.\ \ \ \ \ \ \ \ \ \ \ \ \ \ \ \ \ \ \ \ \ \ \ \  \ (\ell(wv)=\ell(w)+\ell(v))
\end{eqnarray*}

Next, we have 
\begin{eqnarray*}
B_w\Lotimes  B_w^{\geq-1}&\cong&\La\Lwotimes\I_{w}\Lwotimes\La\Lotimes \sigma^{\geq-1}(\La\Lwotimes\I_{v}\Lwotimes\La)\\
&\cong&\La\Lwotimes\I_{w}\Lwotimes \sigma^{\geq-1}(\La\Lwotimes\I_{v}\Lwotimes\La)\\
&\cong&\La\Lwotimes\I_{w}\Lwotimes \I_{v}\Lwotimes\La\ \ \ \ \ \ \ \ \ \ \ \ \ \ \ \ \ \ \ \ \ \  \  (\mathrm{Lemma}\ \ref{quasi-iso})\\ 
&\cong&\La\Lwotimes\I_{wv}\Lwotimes\La\ \ \ \ \ \ \ \ \ \ \ \ \ \ \ \ \ \ \ \ \ \ \ \ \ \ \ \  \  (\mathrm{Proposition}\ \ref{affine braid})\\
&=&B_{wv}.
\end{eqnarray*}

Thus Lemma \ref{homology2} implies that $H^{-2}(B_w\Lotimes B_w^{\geq-1})\cong {}_1(I_{wv})_\nu$.

On the other hand, take the triangle 

$$\xymatrix@C30pt@R30pt{\cdots\ar[r]&B_w^{\leq-2}\ar[r]^{}& B_w\ar[r]^{ }&B_w^{\geq-1}\ar[r]& (B_w^{\leq-2})[1]\ar[r]&\cdots.}$$ 

Then, applying the functor $-\Lotimes B_v^{\geq-1}$ to the triangle, we have the triangle 
$$\xymatrix@C15pt@R30pt{\cdots\ar[r]&B_w^{\leq-2}\Lotimes B_v^{\geq-1}\ar[r]^{}& B_w\Lotimes B_v^{\geq-1}\ar[r]^{ }&B_w^{\geq-1}\Lotimes B_v^{\geq-1}\ar[r]& B_w^{\leq-2}\Lotimes B_v^{\geq-1}[1]\ar[r]&\cdots.}$$

Taking the homology, we have the following long exact sequence 

\[
\xymatrix@C15pt@R5pt
{  0 \ar[r]^{ } &  H^{-2}(B_w^{\leq-2}\Lotimes B_v^{\geq-1})\ar[r]^{h} &H^{-2}(B_w\Lotimes B_v^{\geq-1})   \ar[r] &H^{-2}(B_w^{\geq-1}\Lotimes B_v^{\geq-1})\ar[r] &\\
\ar[r]^{ } &  H^{-1}(B_w^{\leq-2}\Lotimes B_v^{\geq-1})\ar[r] &H^{-1}(B_w\Lotimes B_v^{\geq-1})   \ar[r]^{u_{-1}\ \ \ } &H^{-1}(B_w^{\geq-1}\Lotimes B_v^{\geq-1})\ar[r] &\\
\ar[r]  &    H^0(B_w^{\leq-2}\Lotimes B_v^{\geq-1}) \ar[r] & H^0(B_w\Lotimes B_v^{\geq-1})\ar[r]^{u_{0}\ \ \ \ }  &H^0(B_w^{\geq-1}\Lotimes B_v^{\geq-1}) \ar[r]  &  0. }
\]

Clearly we have 
$H^i(B_w^{\leq-2}\Lotimes B_v^{\geq-1})\cong0$ for $i=0,-1$. 
Hence $u_{-1}$ and $u_{0}$ are isomorphisms. 
Moreover from the above two equalities, 
$h$ is an isomorphism and hence $H^{-2}(B_w^{\geq-1}\Lotimes B_v^{\geq-1})\cong 0$. 
Therefore we have 
\begin{eqnarray*}
B_w^{\geq-1}\Lotimes B_v^{\geq-1}&\cong&\sigma^{\geq-1}(B_w^{\geq-1}\Lotimes B_v^{\geq-1})\\
&\cong&\sigma^{\geq-1}(B_w\Lotimes B_v^{\geq-1})\\
&\cong&\sigma^{\geq-1}(B_{wv}).
\end{eqnarray*}
Thus we get the conclusion.
\end{proof}

Using Proposition \ref{braid condition1}, we obtain the following consequence.

\begin{corollary}\label{braid condition2}
$T_i$ $(i\in \Delta_0^{\rm f})$ satisfy the following braid relations in $\Db(\mod\La^{\rm e})$

\[\left\{\begin{array}{ll}
\ T_i\Lotimes T_{j}\cong T_{j}\Lotimes T_i & \mbox{if no edge between $i$ and $j$ in $\Delta^{\rm f}$},\\
\ T_i\Lotimes T_{j}\Lotimes T_i\cong T_{j}\Lotimes T_i\Lotimes T_{j} & \mbox{if there is an edge $i\stackrel{ }{\mbox{---}}j$ in $\Delta^{\rm f}$},\\
\  T_i\Lotimes T_{j}\Lotimes T_i\Lotimes T_{j}\cong T_{j}\Lotimes T_i\Lotimes T_{j}\Lotimes T_{i} & \mbox{if there is an edge $i\stackrel{4}{\mbox{---}}j$ in $\Delta^{\rm f}$}.\end{array}\right.\]
\end{corollary}

\begin{proof}
We will show the first statement. 
From Proposition \ref{braid condition1}, we have 
$T_i\Lotimes T_{j}\cong\sigma^{\geq-1}(\La\Lwotimes\I_{t_it_j}\Lwotimes\La)$ and 
$T_j\Lotimes T_{i}\cong\sigma^{\geq-1}(\La\Lwotimes\I_{t_jt_i}\Lwotimes\La)$. 
Because we have $t_it_j=t_jt_i$, we have $\I_{t_it_j}=\I_{t_jt_i}$. Therefore 
we conclude $T_i\Lotimes T_{j}\cong T_j\Lotimes T_{i}$. 

By applying Proposition \ref{braid condition1} repeatedly, 
the second and third statements can be shown similarly. 
\end{proof}

Finally we give the following terminology.

\begin{definition}\label{def two-sided set}
We denote by $\langle {T}_i\ |\ i\in \Delta_0^{\rm f} \rangle$ the set of 
two-sided tilting complexes of $\La^{\rm e}$ which can be written as 
$$T_{i_1}^{\epsilon_{i_1}}\Lotimes T_{i_2}^{\epsilon_{i_2}}\Lotimes\cdots \Lotimes T_{i_k}^{\epsilon_{i_k}}\in\Db(\mod\La^{\rm e}),$$
where $i_1,\ldots, i_k\in \Delta_0^{\rm f}$ and $\epsilon_{i_j}\in\{\pm 1\}$. 
Then, for $a=a_{i_1}^{\epsilon_{i_1}}\cdots a_{i_k}^{\epsilon_{i_k}}\in B_{\Delta^{\rm f}}$, 
we define  
$${T}_{a}:=T_{i_1}^{\epsilon_{i_1}}\Lotimes T_{i_2}^{\epsilon_{i_2}}\Lotimes\cdots \Lotimes T_{i_k}^{\epsilon_{i_k}},$$
$$\bmu_{a}:=\bmu_{i_1}^{\epsilon_{i_1}}\circ\cdots \circ \bmu_{i_k}^{\epsilon_{i_k}}(\La).$$ 
\end{definition}

Then the next proposition shows that the left action of ${T}_{i}$ (respectively, ${T}_{i}^{-1}$) gives mutation $\bmu_{i}^+$ (respectively, $\bmu_{i}^-$) in $\Db(\mod\La)$.  

\begin{proposition}\label{operate}
\begin{itemize}
\item[(a)] There is a group homomorphism
$$B_{\Delta^{\rm f}}\to\langle {T}_i\ |\ i\in \Delta_0^{\rm f} \rangle,\ a\mapsto T_a.$$
\item[(b)] We have an isomorphism ${T_a}\cong \bmu_a(\La)$ in $\Db(\mod \La)$.
\end{itemize}
\end{proposition}

\begin{proof} 
(a) follows from Corollary \ref{braid condition2}.

(b) By Lemma \ref{two-sided lemma}, the statement is clear if $a=a_i^{\epsilon_i}$ for any $i\in \Delta_0^{\rm f}$ and $T_i^{\epsilon_i}\cong\bmu_i^{\epsilon_i}(\La)$. 
We will show $T_i^{\epsilon_i} \Lotimes T_j^{\epsilon_j}\cong \bmu_i^{\epsilon_i}\circ\bmu_j^{\epsilon_j}(\La)$ and then the assertion follows from an obvious induction. 
Since mutation is preserved by an equivalence, 
we have $T_i^{\epsilon_i} \Lotimes T_j^{\epsilon_j}\cong
\bmu_i^{\epsilon_i}(\La)\Lotimes T_j^{\epsilon_j}\cong
\bmu_i^{\epsilon_i}(\La\Lotimes T_j^{\epsilon_j})\cong
\bmu_i^{\epsilon_i}\circ\bmu_j^{\epsilon_j}(\La)$.
Thus the assertion holds. 
\end{proof}

\begin{theorem}\label{group surj}
There is a group isomorphism
$$B_{\Delta^{\rm f}}\to\langle {T}_i\ |\ i\in \Delta_0^{\rm f} \rangle,\ a\mapsto T_a,$$
which gives a bijection between $B_{\Delta^{\rm f}}$ and $\tilt\La$.
\end{theorem}

\begin{proof}
This follows from Theorem \ref{tilt-bij} and Proposition \ref{operate}.
\end{proof}


\section{Derived Picard groups}

The notion of derived Picard groups was introduced by Rouquier-Zimmermann and Yekutieli \cite{RZ,Y}, which is the group of auto-equivalences consisting of standard functors modulo functorial isomorphisms.  
For example, those of hereditary algebras \cite{MY}, commutative algebras \cite{Y} and Brauer tree algebras \cite{RZ,SZ} have been investigated. 
The aim of this subsection is to determine the derived Picard group of $\La$ for type $\type{A}_n$ and $\type{D}_n$. 

We follow Notation \ref{main notation}. First we recall the definition of the derived Picard group as follows.

\begin{definition}\label{def dpg} 
The \emph{derived Picard group} $\dpg(\La)$ of $\La$ is the group of isomorphism classes of two-sided tilting complexes of $\Db(\mod\La^{\rm e})$.  
The identity element is $\La$ and the product of the classes of $X$ and $Y$ is given by $X\Lotimes Y$. 
This is equivalent to say the group of auto-equivalences consisting of standard functors modulo functorial isomorphisms. 
\end{definition}

Let $\Aut(\La)$ be the group of automorphisms of $\La$ and 
$\Inn(\La)$ the subgroup consisting of inner automorphisms which is defined by $x\mapsto \lambda x\lambda^{-1}$ ($\lambda\in\Lambda^\times$) for $x\in\La$. 
Moreover the outer automorphisms is defined by 
$\Out(\La)=\Aut(\La)/\Inn(\La)$. 

To give our result, we recall the following well-known result (see, for example \cite[Theorem 3.4.1]{DK} and \cite[Theorem 11.1.7]{HGK} for a more general case). 

\begin{lemma}\label{Inn}
Let $\{f_1,\cdots,f_n\}$ be a complete set of orthogonal primitive idempotents of $\Lambda$. Then there exist $\lambda\in\Lambda^\times$ and unique $\rho\in\mathfrak{S}_n$ such that $\lambda f_i\lambda^{-1}=e_{\rho(i)}$ for any $i$. 
\end{lemma}

For the convenience of the reader, we give a proof of the lemma.

\begin{proof}
Because we have $\bigoplus_{i=1}^ne_i\La =\Lambda=\bigoplus_{i=1}^nf_i\La$, 
there exists $\rho\in\mathfrak{S}_n$ such that 
$f_i\Lambda\cong e_{\rho(i)}\La$ for any $i$, which is unique with respect to
the idempotents. 
Since we have $\Hom_\La(f_i\La,e_{\rho(i)}\La)\cong e_{\rho(i)}\La f_i$ and 
 $\Hom_\La(e_{\rho(i)}\La,f_i\La)\cong f_i\La e_{\rho(i)}$, 
 there exist $\lambda_i\in e_{\rho(i)}\La f_i$ and $\gamma_i\in f_i\La e_{\rho(i)}$ such that $\lambda_i\gamma_i=e_{\rho(i)}$ and $\gamma_i\lambda_i=f_i$. 
Let $\lambda:=\sum_{i=1}^n\lambda_i$ and $\gamma:=\sum_{i=1}^n\gamma_i$. 
Then we have $\lambda\gamma=1=\gamma\lambda$ and $\lambda f_i=\lambda_i=e_{\rho(i)}\lambda $.
\end{proof}

Then we divide the situation into the following two cases.

(Case I). The Nakayama permutation of $\La$ is the identity 
and hence $\Delta=\Delta^{\rm f}$.

(Case II). The Nakayama permutation of $\La$ is not the identity 
and hence $\Delta\neq\Delta^{\rm f}$.

Then we define the action of $\Out(\La)$ on $B_{\Delta^{\rm f}}$ as follows.

\begin{definition}
By Lemma \ref{Inn}, for $\phi\in\Aut(\La)$, there exist 
$\lambda\in\La^\times$ and $\rho^\phi\in\mathfrak{S}_{n}$ such that  
$\phi(e_i)= \lambda e_{\rho^\phi(i)}\lambda^{-1}$, which  
admits a group homomorphism 
$$\Out(\La)\to \mathfrak{S}_{n}, \phi\mapsto\rho^\phi.$$

Then we act $\Out(\La)$ on $B_{\Delta^{\rm f}}$ as follows

\[
\Out(\La)\times B_{\Delta^{\rm f}}\to B_{\Delta^{\rm f}}, (\phi,a)\mapsto a^\phi:=
\left\{\begin{array}{ll}
a_{\rho^\phi(i_1)}^{\epsilon_{i_1}}\cdots a_{\rho^\phi(i_k)}^{\epsilon_{i_k}} & (\mbox{Case I}),\\
a_{i_1}^{\epsilon_{i_1}}\cdots a_{i_k}^{\epsilon_{i_k}} & (\mbox{Case II})\\
\end{array}\right.\]
for an element $a=a_{i_1}^{\epsilon_{i_1}}\cdots a_{i_k}^{\epsilon_{i_k}}\in B_{\Delta^{\rm f}}$.

Then, for $(\phi,a),(\phi',a')\in\Out(\La)\times B_{\Delta^{\rm f}}$, we define the multiplication by  
$$(\phi,a)\cdot(\phi',a')=(\phi\phi',a^{\phi'}a')$$
and define the semidirect product $\Out(\La)\ltimes B_{\Delta^{\rm f}}$.
\end{definition}

Then we will show the following theorem.

\begin{theorem}\label{dpg thm}
Let $\La$ be a preprojective algebra of type $\type{A}_n$ or $\type{D}_n$. 
There is a group isomorphism 
$$\Theta:\Out(\La)\ltimes B_{\Delta^{\rm f}}\to\dpg(\Lambda),\ (\phi,a)\mapsto {}_\phi\La\Lotimes{T}_{a}.$$
\end{theorem}

For a proof, we recall the following basic result (see for example \cite[Proposition 2.3]{RZ}).

\begin{lemma}\label{RZ left}
Let $T$ and $T'$ be two-sided tilting complexes in $\Db(\mod\La^{\rm e})$. 
The restriction of $T$ and $T'$ to $\Db(\mod\La)$ are isomorphic if and only if there exists 
$\phi\in\Out(\La)$ such that 
$$T'\cong {}_\phi\La\Lotimes T.$$ 
\end{lemma}

In the rest of this subsection, we will show that the above $\Theta$ is a group homomorphism.

\begin{lemma}\label{graph auto}
Let $\phi\in\Out(\La)$. 
Then $\rho^\phi$ gives a graph automorphism of $\Delta$.
\end{lemma} 

\begin{proof}
We write $\rho^\phi=\rho$ for simplicity. 
It is enough to show 
$$e_i(\rad\La/\text{rad}^2\La)e_j\cong e_{\rho(i)}(\rad\La/\text{rad}^2\La)e_{\rho(j)}$$
(see, for example \cite[III.Lemma 2.12]{ASS},\cite[section 11]{HGK}).
Since $\phi\in\Out(\La)$ gives $\phi(e_i)=e_{\rho(i)}$, we have 
$$\phi(e_i(\rad\La/\text{rad}^2\La)e_j)\cong e_{\rho(i)}(\rad\La/\text{rad}^2\La)e_{\rho(j)}.$$
This completes the proof.
\end{proof}

Note that $\rho^\phi$ does not necessarily coincide with the Nakayama permutation in general.

Next, we give the following easy lemma. 

\begin{lemma}\label{rho iso}
For any $\phi\in \Out(\wLa)$ and $i\in {\Delta}_0$, 
 we have an isomorphism in $\Db(\mod\wLa^{\rm e})$ 
$${}_{\phi^{-1}}({\I_{i}})_{\phi^{-1}}\cong{\I}_{\rho^\phi(i)}.$$
\end{lemma} 

\begin{proof}
Since $\I_i=\wLa(1-e_i)\wLa$, 
the map $\phi:\I_i\to{\I}_{\rho^\phi(i)}$ given by $x\mapsto \phi(x)$ 
gives an isomorphism ${}_{\phi^{-1}}({\I_{i}})_{\phi^{-1}}\cong{\I}_{\rho^\phi(i)}$ of $\wLa^{\rm e}$-modules.
\end{proof}

Moreover, we use the following result. 

\begin{proposition}\label{restriction}
For any $\phi\in\Aut(\La)$, there exists $\tilde{\phi}\in\Aut(\wLa)$ which makes the following diagram commutative 
$$\xymatrix@C45pt@R15pt{
\wLa\ar[d]_{\rm nat.} \ar[r]^{\tilde{\phi}}   &\ar[d]^{\rm nat.}\wLa  \\
\La \ar[r]^{\phi}  &\La. }$$
In particular, for the above map $\tilde{\phi}$, the map $\tilde{\phi}:{}_{\phi^{-1}}\La_1\to{}_1\La_{\tilde{\phi}}, x\mapsto\tilde{\phi}(x)$ is an isomorphism 
in $\Db(\mod(\La^{\op}\otimes_K\wLa))$.
\end{proposition}

\begin{proof}
We will show the first statement in the next section and the second statement easily follows from the first one.
\end{proof}

(Case I). First, assume that the Nakayama permutation of $\La$ is the identity.

Then we give the following observation.

\begin{lemma}\label{bimo1}
Let $\phi\in\Out(\La)$ and $i\in \Delta_0$ $(=\Delta^{\rm f}_0)$. Then we have isomorphisms in $\Db(\mod\Lambda^{\rm e})$.
$${}_{\phi^{-1}}({T_{i}})_{\phi^{-1}}\cong{T}_{\rho^\phi(i)}\ 
\textnormal{and}\ {}_{\phi^{-1}}({T}_{i}^{-1})_{\phi^{-1}}\cong{T}_{\rho^\phi(i)}^{-1}
.$$
\end{lemma}

\begin{proof}
We will show the first statement and the second statement easily follows from the first one.
Recall that $T_i=\sigma^{\geq-1}(B_i)$, where 
$B_i=\La\Lwotimes\I_{i}\Lwotimes\La.$ 
Hence it is enough to show that ${}_{\phi^{-1}}({B_{i}})_{\phi^{-1}}\cong B_{\rho(i)}$. 

Then, we have 
\begin{eqnarray*}
{}_{\phi^{-1}}({B_{i}})_{\phi^{-1}}&\cong&{}_{\phi^{-1}}\La\Lwotimes\I_{i}\Lwotimes\La_{\phi^{-1}}\\
&\cong&(\La{}_{\tilde{\phi}})\Lwotimes\I_{i}\Lwotimes({}_{\tilde{\phi}}\La)\ \ \ \ \  \ \ \ \ (\textnormal{Proposition}\ \ref{restriction})\\
&\cong&(\La\Lwotimes\wLa{}_{\tilde{\phi}})\Lwotimes\I_{i}\Lwotimes({}_{\tilde{\phi}}\wLa\Lwotimes\La)\\
&\cong&\La\Lwotimes({}_{\tilde{\phi}^{-1}}\wLa\Lwotimes\I_{i}\Lwotimes\wLa{}_{\tilde{\phi}^{-1}})\Lwotimes\La\\
&\cong&\La\Lwotimes{}_{\tilde{\phi}^{-1}}(\I_{i}){}_{{\tilde{\phi}}^{-1}}\Lwotimes\La\\
&\cong&\La\Lwotimes\I_{\rho^\phi(i)}\Lwotimes\La\ \ \ \  \ \ \ \ \ \ \ \ (\textnormal{Lemma}\ \ref{rho iso})\\
&\cong&B_{\rho^\phi(i)}.
\end{eqnarray*}

Thus we get the conclusion.
\end{proof}

(Case II). Next, assume that the Nakayama permutation of $\La$ is not the identity.

Then we have the following lemma.
\begin{lemma}\label{bimo2}
Let $\phi\in\Out(\La)$ and $i\in \Delta^{\rm f}$. 
Then we have isomorphisms in $\Db(\mod\Lambda^{\rm e})$
$${}_{\phi^{-1}}({T_{i}})_{\phi^{-1}}\cong{T}_{i}\ 
\textnormal{and}\ {}_{\phi^{-1}}({T}_{i}^{-1})_{\phi^{-1}}\cong{T}_{i}^{-1}
.$$
\end{lemma}

\begin{proof}
By Lemma \ref{graph auto}, $\rho^\phi$ gives a graph automorphism of $\Delta$. 
On the other hand, $t_i$ (defined in subsection \ref{relation}) is given by an orbit of the Nakayama permutation,  so that it is invariant by a graph automorphism.  
Therefore, by the same argument of Lemma \ref{rho iso}, we get
$${}_{\phi^{-1}}(\I_{t_i})_{\phi^{-1}}\cong\I_{t_i}.$$ 
Thus, the same argument of Lemma \ref{bimo1} implies the assertion.
\end{proof}

From now on, we let $\otimes:=\Lotimes$ for simplicity. Then one can easily show the following lemma. 

\begin{lemma}\label{bimo3}
For any $\phi\in\Out(\La)$ and $a\in B_{\Delta^{\rm f}}$, 
we have an isomorphism in $\Db(\mod\Lambda^{\rm e})$
$${}_{\phi^{-1}}({T_a})_{\phi^{-1}}\cong{T}_{a^{\phi}}.$$
\end{lemma}

\begin{proof}
We write $\rho^\phi=\rho$ for simplicity. 
Let $a=a_{i_1}^{\epsilon_{i_1}}\cdots a_{i_k}^{\epsilon_{i_k}}\in B_{\Delta^{\rm f}}$. 
We first consider the case I.
Using Lemma \ref{bimo1}, we have isomorphisms
\begin{eqnarray*}
{T}_{a^{\phi}}&\cong&{T}_{\rho(i_1)}^{\epsilon_{i_1}}\otimes\cdots\otimes{T}_{\rho(i_k)}^{\epsilon_{i_k}}\\
&\cong&({}_{\phi^{-1}}\La\otimes{T}_{i_1}^{\epsilon_{i_1}}\otimes{}\La_{{\phi}^{-1}})\otimes\cdots\otimes({}_{\phi^{-1}}\La\otimes{T}_{i_k}^{\epsilon_{i_k}}\otimes\La_{{\phi}^{-1}})\\
&\cong&{}_{\phi^{-1}}\La\otimes({T}_{i_1}^{\epsilon_{i_1}}\otimes\cdots\otimes{T}_{i_k}^{\epsilon_{i_k}})\otimes{}\La_{{\phi}^{-1}}\\ 
&\cong&{}_{\phi^{-1}}({T_a})_{\phi^{-1}}.
\end{eqnarray*}

The proof of the case II can be shown similarly.
\end{proof}

Finally, we give a proof of Theorem \ref{dpg thm}.

\begin{proof}[Proof of Theorem \ref{dpg thm}]
For $(\phi,a),(\phi',a')\in\Out(\La)\ltimes B_{\Delta^{\rm f}}$, we have
$(\phi,a)\cdot(\phi',a')=(\phi\phi',a^{\phi'}a')$. 
Using Lemma \ref{bimo3}, 
we have 
\begin{eqnarray*}
\Theta(\phi\phi',a^{\phi'}a')&\cong&{}_{\phi\phi'}\La\otimes{T}_{a^{\phi'}a'}\\
&\cong&{}_{\phi}\La\otimes{}_{\phi'}\La\otimes{T}_{a^{\phi'}}\otimes{T}_{a'}\\
&\cong&{}_{\phi}\La\otimes{}_{\phi'}\La\otimes({}_{{(\phi')}^{-1}}\La\otimes{T}_a\otimes\La_{{(\phi')}^{-1}})\otimes{T}_{a'}\\
&\cong&{}_{\phi}\La\otimes{T}_a\otimes{}_{\phi'}\La\otimes{T}_{a'}\\
&\cong&\Theta(\phi,a)\Theta(\phi',a').
\end{eqnarray*}

Thus the map is a group homomorphism. 

We will show the injectivity. 
Assume that $\Theta(\phi,a)={}_{\phi}\La_1\otimes{T_a}\cong\La$ in $\Db(\mod\La^{\rm e})$. 
Then Theorem \ref{tilt-bij} implies $a=\id$. 
Thus we get ${}_{\phi}\La_1\cong\La$ and hence $\phi\in\Inn(\La)$.  

Next we will show the surjectivity. 
Take $X$ in $\dpg(\Lambda)$. 
Then Theorem \ref{tilt-bij} implies that there exists $a\in B_{\Delta^{\rm f}}$ such that ${T_a}\cong X$ in $\Db(\mod\La)$. 
Then, by Lemma \ref{RZ left}, there exists $\phi\in\Out(\La)$ such that ${}_{\phi}\La_1\otimes{T_a}\cong X$ in $\Db(\mod\La^{\rm e})$.  
\end{proof}


\section{Automorphism groups}
In this section, we give a proof of Proposition  \ref{restriction}. 
We divide the  situation into type $\type{A}_n$ and $\type{D}_n$.
We believe that a similar result holds for $\type{E}_n$ ($n=6,7,8$), though we did not check it because of the difficulty of the calculation of automorphism groups for this type. 

\subsection{The case of type $\type{A}_{n}$}
Let $\La$ be a preprojective algebra of $\type{A}_{n}$ which is given by the following quiver

$$\xymatrix@C50pt@R20pt{
1 \ar[r]^{a_1} &2 \ar[r]^{a_2}\ar@<1ex>[l]^{b_{n-1}}&\ar@<1ex>[l]^{b_{n-2}}\cdots\ar[r]^{a_{n-2}}&n-1 \ar@<1ex>[l]^{b_{2}}\ar[r]^{a_{n-1}}& \ar@<1ex>[l]^{b_1} n.} $$

Let $p$ be the automorphism of $\La$ defined by 
$p(e_i)=e_{n+1-i}$, 
$p(a_i):=b_{i}$ and  $p(b_i):=a_{i}$. Then we have the following result by Iyama \cite{I}.

\begin{proposition}\cite[6.2.2]{I}\label{fix arrowA}
Let $\La$ be a preprojective algebra of type $\type{A}_{n}$ and 
$H:=\{g\in\Aut(\Lambda)\ |\ g$ fixes any $a_i$ and $e_i$ $\}$. Then 
\begin{itemize}
\item[(a)] $\Aut(\La)=\langle\Inn(\Lambda), p, H\rangle.$
\item[(b)] Let $m$ be the maximal integer which does not exceed $n/2$. 
For any $f\in H$, there exist $k_1\in K^\times$ and $k_j\in K$ ($1<j\leq m$) such that 
$$f(b_i)=\sum_{j=1}^mk_j(b_ia_i)^{j-1}b_i\ \ \ (1\leq i\leq n-1).$$
\end{itemize}
\end{proposition}
Then we give a proof of Proposition \ref{restriction} as follows.

\begin{proof}
Let $\widetilde{\La}$ be the preprojective algebra of the following quiver
$$\xymatrix@C50pt@R20pt{&&0 \ar@<-1ex>[lldd]_{a_0}\ar[rrdd]_{b_0}&&\\
&&&&\\
1\ar[rruu]_{b_n} \ar[r]^{a_1} &2 \ar[r]^{a_2}\ar@<1ex>[l]^{b_{n-1}}&\ar@<1ex>[l]^{b_{n-2}}\cdots\ar[r]^{a_{n-2}}&n-1 \ar@<1ex>[l]^{b_{2}}\ar[r]^{a_{n-1}}& \ar@<1ex>[l]^{b_1} \ar@<-1ex>[lluu]_{a_n}n.} $$

The statement for $f\in \Inn(\La)$ and $p$ is clear. 
For $f\in H$, we define $\widetilde{f}\in\Aut(\wLa)$ by 
$\widetilde{f}(e_i):=e_i$, $\widetilde{f}(a_i):=a_i$ and $\widetilde{f}(b_i):=\sum_{j=1}^mk_j(b_ia_i)^{j-1}b_i$ for any $0\leq i\leq n$. 
Then it satisfies the commutative relations, and this gives a desired morphism.
\end{proof}

\subsection{Automorphisms of $\type{D}$}

Let $\La$ be a preprojective algebra of $\type{D}_{n+1}$ $(n\geq 3)$ given by the following quiver 

$$\xymatrix@C50pt@R20pt{&&&-1\ar@<1ex>[d]^{b_{-1}}&\\
n \ar[r]^{a_{n-1}} &n-1 \ar[r]^{a_{n-2}}\ar@<1ex>[l]^{b_{n-1}}&\ar@<1ex>[l]^{b_{n-2}}\cdots\ar[r]^{a_{2}}&2 \ar@<1ex>[l]^{b_{2}}\ar[r]^{a_{1}}\ar[u]^{a_{-1}}& \ar@<1ex>[l]^{b_{1}}1.} $$

In the case of $\type{D}_{4}$, let $p$ be the automorphism of $\La$
with $p^3=\id$ defined by 
$p(a_1)=a_{-1}$,  $p(a_{-1})=b_{2}$, 
$p(b_1)=b_{-1}$,  $p(b_{-1})=a_{2}$, 
$p(e_1)=e_{-1}$, $p(e_{-1})=e_{3}$ and $p(e_2)=e_2$. 
In the case of $\type{D}_{n+1}$ with $n\geq 4$, we let $p=\id$.

We recall some properties from \cite{IRRT}.
The indecomposable projective module $e_{\pm 1}\La$ is given by

\begin{equation}\label{P_1}{\tiny\xymatrix@R1em@C2.5em{
\pm1\ar[d]&&&&&\\
2\ar[r]\ar[d]&\mp1\ar[d]&&&&\\
3\ar[r]\ar[d]&2\ar[r]\ar[d]&\pm1\ar[d]&&&\\
\raisebox{-2pt}[8pt][3pt]{\vdots}\ar[d]&\raisebox{-2pt}[8pt][3pt]{\vdots}\ar[d]&\raisebox{-2pt}[8pt][3pt]{\vdots}\ar[d]&\raisebox{-2pt}{$\ddots$}&&\\
n{-}2\ar[r]\ar[d]&n{-}3\ar[r]\ar[d]&n{-}4\ar[r]\ar[d]&\cdots\ar[r]&\pm({-}1)^{n+1}\ar[d]&\\
n{-}1\ar[r]\ar[d]&n{-}2\ar[r]\ar[d]&n{-}3\ar[r]\ar[d]&\cdots\ar[r]&2\ar[r]\ar[d]&\mp({-}1)^{n+1}\ar[d]&\\
n{}\ar[r]&n{-}1\ar[r]&n{-}2\ar[r]&\cdots\ar[r]&3\ar[r]&2\ar[r]&\pm({-}1)^{n+1}
}}\end{equation}
where each number $i$ shows a $K$-vector space $K$ lying on the vertex $i$, and each arrow is the identity map of $K$. 
Moreover, let $\alpha$ and $\beta$ be scalars satisfying $\alpha+\beta=1$.
Then $e_i\La$ for $2\leq i\leq n$ is given by
\begin{equation}\label{P_i}
{\tiny\begin{xy}
0;<6pt,0pt>:<0pt,6pt>::
(0,0)*+{n}="11",
(0,5)*+{n{-}1}="12",
(0,10)*+{\raisebox{-2pt}[7pt][4pt]{\vdots}}="13",
(0,15)*+{i{+}2}="14",
(0,20)*+{i{+}1}="15",
(0,25)*+{i}="16",
(5,0)*+{n{-}1}="21",
(5,5)*+{n{-}2}="22",
(5,10)*+{\raisebox{-2pt}[7pt][4pt]{\vdots}}="23",
(5,15)*+{i{+}1}="24",
(5,20)*+{i}="25",
(5,25)*+{i{-}1}="26",
(10,0)*+{\cdots}="31",
(10,5)*+{\cdots}="32",
(10,10)*+{\cdots}="33",
(10,15)*+{\cdots}="34",
(10,20)*+{\cdots}="35",
(10,25)*+{\cdots}="36",
(16,0)*+{n{-}i{+}1}="41",
(16,5)*+{n{-}i}="42",
(16,10)*+{\raisebox{-2pt}[7pt][4pt]{\vdots}}="43",
(16,15)*+{4}="44",
(16,20)*+{3}="45",
(16,25)*+{2}="46",
(23,0)*+{n{-}i}="51",
(23,5)*+{n{-}i{-}1}="52",
(23,10)*+{\raisebox{-2pt}[7pt][4pt]{\vdots}}="53",
(23,15)*+{3}="54",
(23,20)*+{2}="55",
(24.5,26)*+{1}="56",
(21.5,24)*+{{-}1}="56-",
(29.5,0)*+{n{-}i{-}1}="61",
(29.5,5)*+{n{-}i{-}2}="62",
(29.5,10)*+{\raisebox{-2pt}[7pt][4pt]{\vdots}}="63",
(29.5,15)*+{2}="64",
(31,21)*+{{-}1}="65",
(28,19)*+{1}="65-",
(29.5,25)*+{2}="66",
(36,0)*+{n{-}i{-}2}="71",
(36,5)*+{n{-}i{-}3}="72",
(36,10)*+{\raisebox{-2pt}[7pt][4pt]{\vdots}}="73",
(37.5,16)*+{1}="74",
(34.5,14)*+{{-}1}="74-",
(36,20)*+{{}2}="75",
(36,25)*+{{}3}="76",
(42,0)*+{n{-}i{-}3}="81",
(42,5)*+{n{-}i{-}4}="82",
(43,12)="83",
(42,10)*+{\raisebox{-2pt}[7pt][4pt]{\vdots}}="83middle",
(41,11)="83-",
(42,15)*+{{}2}="84",
(42,20)*+{{}3}="85",
(42,25)*+{{}4}="86",
(47,0)*+{\cdots}="91",
(47,5)*+{\cdots}="92",
(47,10)*+{\cdots}="93",
(47,15)*+{\cdots}="94",
(47,20)*+{\cdots}="95",
(47,25)*+{\cdots}="96",
(52,0)*+{i-1}="a1",
(52,5)*+{i}="a2",
(52,10)*+{\raisebox{-2pt}[7pt][4pt]{\vdots}}="a3",
(52,15)*+{n-3}="a4",
(52,20)*+{n-2}="a5",
(52,25)*+{n-1}="a6",
(57,0)*+{i}="b1",
(57,5)*+{i+1}="b2",
(57,10)*+{\raisebox{-2pt}[7pt][4pt]{\vdots}}="b3",
(57,15)*+{n-2}="b4",
(57,20)*+{n-1}="b5",
(57,25)*+{n}="b6",
\ar@<0pt>"12";"11",
\ar@<0pt>"13";"12",
\ar@<0pt>"14";"13",
\ar@<0pt>"15";"14",
\ar@<0pt>"16";"15",
\ar@<0pt>"22";"21",
\ar@<0pt>"23";"22",
\ar@<0pt>"24";"23",
\ar@<0pt>"25";"24",
\ar@<0pt>"26";"25",
\ar@<0pt>"42";"41",
\ar@<0pt>"43";"42",
\ar@<0pt>"44";"43",
\ar@<0pt>"45";"44",
\ar@<0pt>"46";"45",
\ar@<0pt>"52";"51",
\ar@<0pt>"53";"52",
\ar@<0pt>"54";"53",
\ar@<0pt>"55";"54",
\ar@<0pt>^(0.55)\alpha"56";"55",
\ar@<0pt>_(0.4)\beta"56-";"55",
\ar@<0pt>"62";"61",
\ar@<0pt>"63";"62",
\ar@<0pt>"64";"63",
\ar@<0pt>^(0.55)\alpha"65";"64",
\ar@<0pt>_(0.4)\beta"65-";"64",
\ar@<0pt>^\beta"66";"65",
\ar@<0pt>_{-\alpha}"66";"65-",
\ar@<0pt>"72";"71",
\ar@<0pt>"73";"72",
\ar@<0pt>^(0.55)\alpha"74";"73",
\ar@<0pt>_(0.4)\beta"74-";"73",
\ar@<0pt>^\beta"75";"74",
\ar@<0pt>_{-\alpha}"75";"74-",
\ar@<0pt>"76";"75",
\ar@<0pt>"82";"81",
\ar@<0pt>"83middle";"82",
\ar@<0pt>^(0.6)\beta"84";"83",
\ar@<0pt>_(0.7){-\alpha}"84";"83-",
\ar@<0pt>"85";"84",
\ar@<0pt>"86";"85",
\ar@<0pt>"a2";"a1",
\ar@<0pt>"a3";"a2",
\ar@<0pt>"a4";"a3",
\ar@<0pt>"a5";"a4",
\ar@<0pt>"a6";"a5",
\ar@<0pt>"b2";"b1",
\ar@<0pt>"b3";"b2",
\ar@<0pt>"b4";"b3",
\ar@<0pt>"b5";"b4",
\ar@<0pt>"b6";"b5",

\ar@<0pt>"11";"21",
\ar@<0pt>"12";"22",
\ar@<0pt>"14";"24",
\ar@<0pt>"15";"25",
\ar@<0pt>"16";"26",
\ar@<0pt>"21";"31",
\ar@<0pt>"22";"32",
\ar@<0pt>"24";"34",
\ar@<0pt>"25";"35",
\ar@<0pt>"26";"36",
\ar@<0pt>"31";"41",
\ar@<0pt>"32";"42",
\ar@<0pt>"34";"44",
\ar@<0pt>"35";"45",
\ar@<0pt>"36";"46",
\ar@<0pt>"41";"51",
\ar@<0pt>"42";"52",
\ar@<0pt>"44";"54",
\ar@<0pt>"45";"55",
\ar@<0pt>"46";"56",
\ar@<0pt>"46";"56-",
\ar@<0pt>"51";"61",
\ar@<0pt>"52";"62",
\ar@<0pt>"54";"64",
\ar@<0pt>"55";"65",
\ar@<0pt>"55";"65-",
\ar@<0pt>"56";"66",
\ar@<0pt>_{-1}"56-";"66",
\ar@<0pt>"61";"71",
\ar@<0pt>"62";"72",
\ar@<0pt>"64";"74",
\ar@<0pt>"64";"74-",
\ar@<0pt>"65";"75",
\ar@<0pt>_{-1}"65-";"75",
\ar@<0pt>"66";"76",
\ar@<0pt>"71";"81",
\ar@<0pt>"72";"82",
\ar@<0pt>"74";"84",
\ar@<0pt>_{-1}"74-";"84",
\ar@<0pt>"75";"85",
\ar@<0pt>"76";"86",
\ar@<0pt>"81";"91",
\ar@<0pt>"82";"92",
\ar@<0pt>"84";"94",
\ar@<0pt>"85";"95",
\ar@<0pt>"86";"96",
\ar@<0pt>"91";"a1",
\ar@<0pt>"92";"a2",
\ar@<0pt>"94";"a4",
\ar@<0pt>"95";"a5",
\ar@<0pt>"96";"a6",
\ar@<0pt>"a1";"b1",
\ar@<0pt>"a2";"b2",
\ar@<0pt>"a4";"b4",
\ar@<0pt>"a5";"b5",
\ar@<0pt>"a6";"b6",
\end{xy}}\end{equation} 
where each number $i$ shows a $K$-vector space $K$ lying on the vertex $i$. 
Each unlabelled arrow is the identity map of $K$ and
each arrow labelled by a scalar $\gamma$ is a linear map multiplying by $\gamma$ \cite[Lemma 6.9]{IRRT}.

The next lemma is a direct consequence of (\ref{P_i}).

\begin{lemma}\label{basis}
\begin{itemize}
\item[(a)] For $3\leq i\leq n$, a basis of $e_i\La e_{i-1}$ is given by 
$$\{a_{i-1}(b_{i-1}a_{i-1})^{j-1},\ \ (a_{i-1}b_{i-1})^{j-1}a_{i-1}a_{i-2}\cdots a_1b_1b_2\cdots b_{i-2}
\ \ (1\leq j\leq n-i+1)\}.$$

In particular, we have  
$e_{i}\Lambda e_{i-1}=a_{i-1}(e_{i-1}\La e_{i-1})$.

\item[(b)] A basis of $e_2\La e_{\pm 1}$ is given by 
$$\{(b_{2}a_{2})^{j-1}a_{\pm 1}\ \ (1\leq j\leq n-1)\}.$$

\item[(c)] For $2\leq i\leq n-1$, a basis of $e_i\La e_{i+1}$ is given by 
$$\{(b_{i}a_{i})^{j-1}b_{i},\ \ (b_{i}a_{i})^{j-1}a_{i-1}a_{i-2}\cdots a_1b_1b_2\cdots b_{i}
\ \ (1\leq j\leq n-i)\}.$$
\end{itemize}
\end{lemma}

Then we will show the following result, which is analogous to Proposition \ref{fix arrowA}.

\begin{proposition}\label{fix arrow}
Let $\La$ be a preprojective algebra of type $\type{D}_{n+1}$ $(n\geq3)$ and 
$H:=\{g\in\Aut(\Lambda)\ |\ g$ fixes $e_1+e_{-1}$, $a_1+a_{-1}$ and any $e_i,a_i$ for $i\geq 2$  $\}$. Then 
\begin{itemize}
\item[(a)] $\Aut(\La)=\langle\Inn(\Lambda), p, H\rangle$. 
\item[(b)] For any $f\in H$, there exist $k_1\in K^\times$ and $k_j\in K$ $(1<j\leq n-2)$ such that  
$$f(b_{1}+b_{-1})=\sum^{n-2}_{j=1}k_j(b_{1}+b_{-1})(b_2a_2)^{j-1}\ \textnormal{and} \ f(b_{i})=\sum^{n-2}_{j=1}k_j(b_ia_i)^{j-1}b_i\ (2\leq i\leq n-1).$$
\end{itemize}
\end{proposition}

\begin{proof}
(a) Let $g\in\Aut(\Lambda)$.
In the case of $\type{D}_{n+1}$ for $n\geq 4$, 
by Lemmas \ref{Inn} and \ref{graph auto}, there exists $\lambda\in\La^\times$ such that $\lambda g(-)\lambda^{-1}\in U$, where $U:=\{g\in\Aut(\Lambda)\ |\ g$ fixes $e_1+e_{-1}$ and 
any $e_i$ for $2\leq i\leq n$  $\}$. 
In the case of $\type{D}_{4}$, 
we have $p^m\circ\lambda g(-)\lambda^{-1}\in U$ for some $m\in\{1,2,3\}$. 
Thus 
it is enough to show that $U$ is generated by $\Inn(\Lambda)$ and $H$.

Let $g\in U$. Lemma \ref{basis} (a) implies that  
$e_{i}\Lambda e_{i-1}=a_{i-1}(e_{i-1}\La e_{i-1})$ for any $3\leq i\leq n$. 
Moreover Lemma \ref{basis} (b) and the relation $b_2a_2=(a_1+a_{-1})(b_1+b_{-1})$ imply  
$e_{2}\La (e_{1}+e_{-1})=(a_1+a_{-1})(e_{1}+e_{-1})\Lambda (e_{1}+e_{-1})$.
Then it is easy to check that we can take $\lambda_i\in (e_i\La e_i)^{\times}$ such that $\lambda_n:=e_n$ and 
$\lambda_{i+1}g(a_i)=a_i\lambda_i$ for any $2\leq i\leq n-1$, and 
$\lambda_{1,-1}\in ((e_1+e_{-1})\La (e_1+e_{-1}))^\times$ such that 
$\lambda_{2}g(a_1+a_{-1})=(a_1+a_{-1})\lambda_{1,-1}$. 
Then, for $\lambda:=\lambda_{1,-1}+\lambda_2+\cdots+\lambda_n\in\La^{\times}$, 
we have $\lambda g(e_i)\lambda^{-1}=e_i$, $\lambda g(a_i)\lambda^{-1}=a_i$ for $i\geq 2$ and 
$\lambda  g(e_1+e_{-1})\lambda^{-1}=e_1+e_{-1}$, $\lambda  g(a_1+a_{-1})\lambda^{-1}=a_1+a_{-1}$. 
Thus we get $\lambda g(-)\lambda^{-1}\in H$.

(b) Let $f\in H$. 
Since $f(b_{n-1})=f(e_{n-1}b_{n-1}e_{n})\in e_{n-1}\La e_n$, 
Lemma \ref{basis} (c) implies that we can write $$f(b_{n-1}) =k_1b_{n-1}+k_1'a_{n-1}a_{n-2}\cdots a_1b_1b_2\cdots b_{n-1}$$ for some $k_1,k_1'\in K$. 
Because we have $a_{n-1}f(b_{n-1})=f(a_{n-1})f(b_{n-1})=f(a_{n-1}b_{n-1})=0$, we get $f(b_{n-1}) =k_1b_{n-1}$ and $k_1\in K^\times$. 
Similarly $f(b_{n-1})a_{n-1}=a_{n-2}f(b_{n-2})$ implies that 
$f(b_{n-2}) =k_1b_{n-2}+k_2(b_{n-2}a_{n-2})b_{n-2}$ for some $k_2\in K$.
Inductively $f(b_{i})a_{i}=a_{i-1}f(b_{i-1})$  $(3\leq i\leq n-1)$ implies that 
$f(b_{i})=\sum^{n-i}_{j=1}k_j(b_ia_i)^{j-1}b_i$ for $2\leq i\leq n-1$ for some $k_j\in K$. 
Therefore, by the relations, it can be written as $$f(b_{i})=\sum^{n-2}_{j=1}k_j(b_ia_i)^{j-1}b_i.$$

Moreover by the relation $f(b_2)a_2=(a_1+a_{-1})f(b_1+b_{-1})$, 
we obtain $f(b_{1}+b_{-1})=\sum^{n-2}_{j=1}k_j(b_1+b_{-1})(b_2a_2)^{j-1}.$
\end{proof}

We consider the (completed) preprojective algebra $\wLa$ of type $\tilde{\type{D}}_{n+1}$ given by the following quiver

$$\xymatrix@C50pt@R25pt{&&&-1\ar@<1ex>[d]^{b_{-1}}&\\
n \ar[r]^{a_{n-1}} &n-1\ar@<0.5ex>[d]^{b_{-(n-1)}} \ar[r]^{a_{n-2}}\ar@<1ex>[l]^{b_{n-1}}&\ar@<1ex>[l]^{b_{n-2}}\cdots\ar[r]^{a_{2}}&2 \ar@<1ex>[l]^{b_{2}}\ar[r]^{a_{1}}\ar[u]^{a_{-1}}& \ar@<1ex>[l]^{b_{1}}1.\\
&0\ar@<0.5ex>[u]^{a_{-(n-1)}}&&&} $$

From now on, for simplicity, we write 
\begin{eqnarray*}
(b_1 [0])&:=&e_1,\\
(b_1 [1])&:=&(b_1a_{-1}),\\
(b_1 [2])&:=&(b_1a_{-1})(b_{-1}a_1),\\
(b_1 [3])&:=&(b_1a_{-1})(b_{-1}a_1)(b_1a_{-1}),\\
\vdots
\end{eqnarray*}
and, similarly, define a path $(b_1 [j])$ of $e_1\La(e_1+e_{-1})$, and  of $e_1\wLa(e_1+e_{-1})$ for $j\geq0$. 
Note that we have either $(b_1 [j])b_1=0$ or $(b_1[j])b_{-1}=0$. 
Moreover, we define $(b_1 [j])b_{\overline{1}}$ by 
$b_{\overline{1}}:=
\begin{cases}
b_1 &\textnormal{if}\ j: \textnormal{even},\\
b_{-1} &\textnormal{if}\ j: \textnormal{odd}.
\end{cases}
$
For example, we have 
$(b_1[3])b_{\overline{1}}=(b_1a_{-1})(b_{-1}a_1)(b_1a_{-1})b_{-1}$. 
Similarly we define $(b_{-1} [j])$ and $(b_{-1} [j])b_{\overline{-1}}:=
\begin{cases}
b_{-1} &\textnormal{if}\ j: \textnormal{even},\\
b_{1} &\textnormal{if}\ j: \textnormal{odd}.
\end{cases}
$

Then using the above terminology, we define a path
$$\tilde{f}_{\infty}(b_1,b_{-1})=\sum^{\infty}_{j=1}c_j(b_{1}[j])+\sum^{\infty}_{j=1}d_j(b_{-1}[j])\in (e_1+e_{-1})\wLa(e_1+e_{-1})$$
for some $c_j,d_j\in K$. 
Similarly we write 
\begin{eqnarray*}
(a_{n-1} [0])&:=&e_n,\\
(a_{n-1} [1])&:=&(a_{n-1}b_{-(n-1)}),\\
(a_{n-1} [2])&:=&(a_{n-1}b_{-(n-1)})(a_{-(n-1)}b_{n-1}),\\
(a_{n-1} [3])&:=&(a_{n-1}b_{-(n-1)})(a_{-(n-1)}b_{n-1})(a_{n-1}b_{-(n-1)}),\\
\vdots
\end{eqnarray*}
and define a path $(a_{n-1} [j])\in e_n\wLa(e_0+e_n)$ for $j\geq0$, and similarly we define a path $(a_{-(n-1)} [j])\in e_0\wLa(e_0+e_n)$.%

Then, we define 
$$\tilde{f}_{\infty}(a_{n-1},a_{-(n-1)}):=\sum^{\infty}_{j=1}c_j(a_{n-1}[j])+\sum^{\infty}_{j=1}d_j(a_{-(n-1)}[j])\in(e_0+e_n)\wLa(e_0+e_n)$$

We will use these terminologies in the next subsections.


\subsection{The case of $\type{D}_{4}$.}
Before dealing with the general case, 
we deal with the case of $\type{D}_{4}$  (i.e. $n=3$) in this subsection, which helps us to know the strategy for general settings. 
By  Proposition \ref{fix arrow}, it is enough to show it for $f\in H$. 

Fix $f\in H$. 
By (\ref{P_1}), 
we have $(b_{\pm 1} [j])=0$ for $j\geq 3$.
Then, because $f(e_1)\in (e_1+e_{-1})\La(e_1+e_{-1})$, 
we can write $f(e_1)=
\sum^{2}_{j=0}c_j(b_i [j])+\sum^{2}_{j=0}d_j(b_{-1}[j])$
for some $c_j,d_j\in K$. Since $f(e_1)$ is an idempotent, we get 
$c_0^2=c_0$ and $d_0^2=d_0$. Without loss of generality, we can choose $c_0=1$ and $d_0=0$ and moreover we have
$$-c_2=c_1d_{1}=d_{2}.$$

Because $f(e_1+e_{-1})=e_1+e_{-1},$ we obtain

$$
f(e_{-1})=e_{-1}-\sum^{2}_{j=1}c_j(b_{1}[j])-\sum^{2}_{j=1}d_j(b_{-1}[j]).
$$

On the other hand, 
by $f(a_1)=f((a_1+a_{-1})e_1)=(a_1+a_{-1})f(e_1)$ and $f(a_{-1})=(a_1+a_{-1})f(e_{-1})$,
we have
$$
f(a_1)=a_1+c_1a_1(b_{1}[1])+d_1a_{-1}(b_{-1}[1]),\ 
 f(a_{-1})=a_{-1}-c_1a_1(b_{1}[1])-d_1a_{-1}(b_{-1}[1]).$$

Furthermore, since $f(b_1)=f(e_1)f(b_1+b_{-1})=f(e_1)k_1(b_1+b_{-1})$ and $f(b_{-1})=f(e_{-1})f(b_1+b_{-1})=f(e_{-1})k_1(b_1+b_{-1})$, we have 
$$f(b_1)=k_1\{ b_1+c_1(b_1[1])b_{-1}+d_1(b_{-1}[1])b_1\}, f(b_{-1})=k_1\{ b_{-1}-c_1(b_1[1])b_{-1}-d_1(b_{-1}[1])b_1\}.$$

Then, by the relations $f(b_1)f(a_1)=0$ and $f(b_{-1})f(a_{-1})=0$, we obtain   
\begin{eqnarray}\label{D_4}
c_1+d_1=0.
\end{eqnarray}

Now we consider the preprojective algebra $\wLa$ of type $\tilde{\type{D}}_{4}$ given by the following quiver

$$\xymatrix@C50pt@R20pt{&&-1\ar@<0.5ex>[d]^{b_{-1}}&&\\
&3\ar[r]^{a_{2}}&2\ar@<0.5ex>[d]^{b_{-2}} \ar@<1ex>[l]^{b_{2}}\ar[r]^{a_{1}}\ar@<0.5ex>[u]^{a_{-1}}& \ar@<1ex>[l]^{b_{1}}1.\\
&&0\ar@<0.5ex>[u]^{a_{-2}}&} $$

Then, using the above $c_1,c_2,d_1,d_2$, 
we give the following correspondence $\tilde{f}$, and we will show that $\tilde{f}$ gives an automorphism of $\wLa$.

$\bf{(i)}$ First we define
\begin{eqnarray*}
\tilde{f}(e_1):=e_1+\tilde{f}_{\infty}(b_1,b_{-1}),\ \ \ 
\tilde{f}(e_{-1}):=e_1-\tilde{f}_{\infty}(b_1,b_{-1}), 
\end{eqnarray*}
where we define $c_j$ and $d_j$ ($j\geq 3$) as follows.
For odd $j$, we let $c_{j}=0=d_j$ for any $j$. 
For even $j$, we define 
$$-c_j=c_1d_{j-1}+c_2c_{j-2}+c_3d_{j-3}+c_4c_{j-4}+\cdots+c_{j-2}c_{2}+c_{j-1}d_{1}=d_{j}.$$

Note that $\tilde{f}(e_1)^2=\tilde{f}(e_1)$, $\tilde{f}(e_{-1})^2=\tilde{f}(e_{-1})$ and $\tilde{f}(e_1)\tilde{f}(e_{-1})=0=\tilde{f}(e_{-1})\tilde{f}(e_1)$.

$\bf{(ii)}$ Secondly, we define  
$\tilde{f}(a_1):=(a_1+a_{-1})\tilde{f}(e_1)$ and $\tilde{f}(a_{-1}):=(a_1+a_{-1})\tilde{f}(e_{-1})$, that is, 
\begin{eqnarray*}
\tilde{f}(a_1)&:=&a_1+\sum^{\infty}_{j=1}c_ja_1(b_{1}[j])+\sum^{\infty}_{j=1}d_ja_{-1}(b_{-1}[j]),\\
\tilde{f}(a_{-1})&:=&a_{-1}-\sum^{\infty}_{j=1}c_ja_1(b_{1}[j]) - \sum^{\infty}_{j=1}d_ja_{-1}(b_{-1}[j]). 
\end{eqnarray*}

$\bf{(iii)}$ Thirdly, 
we define $\tilde{f}(b_1):=\tilde{f}(e_1)k_1(b_1+b_{-1})$ and $\tilde{f}(b_{-1}):=\tilde{f}(e_{-1})k_1(b_1+b_{-1})$,  
that is, 
\begin{eqnarray*}
\tilde{f}(b_1)&:=&k_1\{b_1+\sum^{\infty}_{j=1}c_j(b_{1}[j])b_{\overline{1}}+\sum^{\infty}_{j=1}d_j(b_{-1}[j])b_{\overline{-1}}\},\\
\tilde{f}(b_{-1})&:=&k_1\{b_{-1}-\sum^{\infty}_{j=1}c_j(b_{1}[j])b_{\overline{1}}-\sum^{\infty}_{j=1}d_j(b_{-1}[j])b_{\overline{-1}}\}. 
\end{eqnarray*}

\begin{remark}Because $c_1+d_1=0$ and $-c_2=c_1d_1=d_2$, we can explicitly describe $\tilde{f}(e_1)$ as follows.
\begin{eqnarray*}
\tilde{f}(e_1)&=&e_1+c_1(b_1[1])+c_2(b_1[2])+c_4(b_1[4])+\cdots +c_{2j}(b_1[2j])+\cdots \\
&+&\ \ d_{1}(b_{-1}[1])+d_{2}(b_{-1}[2])+d_{4}(b_{-1}[4])+\cdots+d_{2j}(b_{-1}[2j])+\cdots
\end{eqnarray*}
where $c_{2j}=-d_{2j}=(-1)^{j+1}C_{j-1}c_1^{2j}$ ($j\geq 1$) and 
$C_{j-1}:= \frac{(2j-2)!}{j!(j-1)!}$ is the Catalan number. 
Then we can get $\tilde{f}(e_1)^2=\tilde{f}(e_1)$ by using the property  
$C_{j+1}=\sum_{i=0}^{j}C_iC_{j-i}$.
\end{remark}

Furthermore, we can check the following lemma.

\begin{lemma}\label{zero relation D_4}
We have $\tilde{f}(b_{\pm 1})\tilde{f}(a_{\pm 1})=0.$
\end{lemma}

\begin{proof} 
We only show it for $\tilde{f}(b_1)\tilde{f}(a_1)$ and the case of $\tilde{f}(b_{-1})\tilde{f}(a_{-1})$ is similar. 
For simplicity, we denote by $\Co(b_{\pm 1} [j])$ the coefficient of $(b_{\pm 1} [j])$ in $\tilde{f}(b_1)\tilde{f}(a_1)$.  

Fix even $m$ with $m\geq 2$. By the direct calculation, we can check

\begin{eqnarray*}
\Co(b_1[m])&=&k_1\{(c_{m-1}+d_{m-1})+c_2(c_{m-3}+d_{m-3})+c_4(c_{m-5}+d_{m-5})+\cdots+c_{m-2}(c_{1}+d_{1})   \} \nonumber \\
\Co(b_1 [m+1])&=&k_1\{c_1(c_{m-1}+d_{m-1})+c_3(c_{m-3}+d_{m-3})+c_5(c_{m-5}+d_{m-5})+\cdots+c_{m-1}(c_{1}+d_{1})   \}\\
\Co(b_{-1}[m])&=&k_1\{d_2(c_{m-3}+d_{m-3})+d_4(c_{m-5}+d_{m-5})+d_6(c_{m-7}+d_{m-7})+\cdots+d_{m-2}(c_{1}+d_{1})   \} \nonumber \\
\Co(b_{-1} [m+1])&=&k_1\{d_1(c_{m-1}+d_{m-1})+d_3(c_{m-3}+d_{m-3})+d_5(c_{m-5}+d_{m-5})+\cdots+d_{m-1}(c_{1}+d_{1})   \}.\\
\end{eqnarray*}
Then because $c_j=0=d_j$ for odd $j\geq3$ and $c_1+d_1=0$
 by (\ref{D_4}), they are zero.
\end{proof}

Then, we obtain the following desired result. 

\begin{proposition}\label{prop D4}
For $f\in H$, we have $\tilde{f}\in\Aut(\wLa)$ such that
$$\xymatrix@C45pt@R15pt{
\wLa\ar[d]_{\rm nat.} \ar[r]^{\tilde{f}}   &\ar[d]^{\rm nat.}\wLa  \\
\La \ar[r]^{f}  &\La}$$
by defining $\tilde{f}$ as follows :

$\bullet$ $\tilde{f}(e_{\pm1}),\tilde{f}(a_{\pm1})$ and $\tilde{f}(b_{\pm1})$ as {\bf (i),(ii),(iii)} and $\tilde{f}(e_2):=e_2$. 
\begin{eqnarray*}
\bullet\ \tilde{f}(e_3):=e_3+\tilde{f}_{\infty}(a_2,a_{-2}),\ \ \ 
\tilde{f}(e_{0}):=e_{0}-\tilde{f}_{\infty}(a_2,a_{-2}),\\
\tilde{f}(b_2):=k_1(b_2+b_{-2})\tilde{f}(e_3),\ \ \ 
\tilde{f}(b_{-2}):=k_1(b_2+b_{-2})\tilde{f}(e_{0}),\\
\tilde{f}(a_2):=\tilde{f}(e_3)(a_2+a_{-2}),\ \ \ 
\tilde{f}(a_{-2}):=\tilde{f}(e_{0})(a_2+a_{-2}).
\end{eqnarray*}
\end{proposition}

\begin{proof}
We will check the following relations

\[\left\{\begin{array}{lll}
\widetilde{f}(b_{\pm 1})\widetilde{f}(a_{\pm 1})=0 &&\textnormal{(a)}\\
 \widetilde{f}(a_{\pm 2})\widetilde{f}(b_{\pm 2})=0
&&\textnormal{(b)}\\
\widetilde{f}(b_{2}+b_{-2})\widetilde{f}(a_{2}+a_{-2})=\widetilde{f}(a_1+a_{-1})\widetilde{f}(b_1+b_{-1}) &&\textnormal{(c)}
 \end{array}\right.\]

(a) This follows from Lemma \ref{zero relation D_4}.

(b) Because the coefficients of $\tilde{f}_{\infty}(a_2,a_{-2})$ is the same as the ones of $\tilde{f}_{\infty}(b_1,b_{-1})$, this follows from (a) by the same calculation.

(c) Because $\widetilde{f}(e_3+e_0)=e_3+e_0$, we get 
$\widetilde{f}(b_{2}+b_{-2})\widetilde{f}(a_{2}+a_{-2})=
k_1(b_{2}+b_{-2})(a_{2}+a_{-2})=
(a_1+a_{-1})k_1(b_1+b_{-1})=
\widetilde{f}(a_1+a_{-1})\widetilde{f}(b_1+b_{-1}).$

Thus we can obtain $\widetilde{f} \in\Aut(\wLa)$. 
The second statement is clear from the definition of $\widetilde{f}$.
\end{proof}

\subsection{The case of $\type{D}_{5}$.}

Next we will deal with the case of $\type{D}_{5}$ (i.e. the case of $n=4$).

Then as the case $\type{D}_{4}$, we can write 
$
f(e_{1})=e_{1}+\sum^{3}_{j=1}c_j(b_{1}[j])+\sum^{3}_{j=1}d_j(b_{-1}[j])$
and $f(e_{-1})=e_{-1}-\sum^{3}_{j=1}c_j(b_{1}[j])-\sum^{3}_{j=1}d_j(b_{-1}[j])
$
for some $c_j,d_j\in K$ such that 
$$ -c_2=c_1d_{1}=d_{2}.$$

Then, similarly, we have
$$
f(a_1)=a_1+\sum^{2}_{j=1}c_ja_1(b_{1}[j])+\sum^{2}_{j=1}d_ja_{-1}(b_{-1}[j]),\ 
 f(a_{-1})=a_{-1}-\sum^{2}_{j=1}c_ja_1(b_{1}[j])-\sum^{2}_{j=1}d_ja_{-1}(b_{-1}[j]).$$
Since 
$f(b_1+b_{-1})=\sum^{2}_{j=1}k_j(b_1+b_{-1})(b_2a_2)^{j-1}
=\sum^{2}_{j=1}k_j\{( b_1[j-1])b_{\overline{1}}+(b_{-1}[j-1])b_{\overline{-1}})\}$, 
we have  
\begin{eqnarray*}
f(b_1)&=&k_1\{b_1+c_1(b_1[1])b_{-1}+c_2(b_1[2])b_1+ 
d_{1}(b_{-1}[1])b_{1}+d_{2}(b_{-1}[2])b_{-1}\}\\
&+&k_2\{(b_1[1])b_{-1}+c_1(b_1[2])b_1+ 
d_{1}(b_{-1}[2])b_{-1}\},\\
f(b_{-1})&=&k_1\{b_{-1}-c_1(b_1[1])b_{-1}-c_2(b_1[2])b_1- 
d_{1}(b_{-1}[1])b_{1}-d_{2}(b_{-1}[2])b_{-1}\}\\
&+&k_2\{(b_{-1}[1])b_{1}-c_1(b_1[2])b_1- 
d_{1}(b_{-1}[2])b_{-1}\}.
\end{eqnarray*}

Then, by the relations $f(b_1)f(a_1)=0$ and $f(b_{-1})f(a_{-1})=0$, we have

\begin{eqnarray}\label{D_5}
c_1+d_1=0\ \textnormal{and}\ k_2=0.
\end{eqnarray}

Next, we consider the factorization of $f$ using $u,v\in\Aut(\La)$, which are defined as follows.  
First, $u$ is defined by $u(e_1):=e_{1}+\sum^{2}_{j=1}c_j(b_{1}[j])+\sum^{2}_{j=1}d_j(b_{-1}[j])$, $u(e_{-1}):=e_{-1}-\sum^{2}_{j=1}c_j(b_{1}[j])-\sum^{2}_{j=1}d_j(b_{-1}[j]),$ 
$u(a_{\pm 1}):=f(a_{\pm 1})$, $u(b_{\pm 1}):=f(b_{\pm 1})$ and
$u(e_i):=f(e_i)$, $u(a_i):=f(a_i)$, $u(b_i):=f(b_i)$ for any $i\geq 2$ (i.e. $u=f$ if $c_3=0=d_3$). 

Second, $v$ is defined by $v(e_1)=e_1+c_3(b_{1}[3])+d_3(b_{-1}[3])$, $v(e_{-1})=e_{-1}-c_3(b_{1}[3])-d_3(b_{-1}[3])$, $v(a_{\pm 1}):=a_{\pm 1}$, $v(b_{\pm 1}):=b_{\pm 1}$
and $v(e_i)=e_i$, $v(a_i)=a_i$, $v(b_i)=b_i$ for any $i\geq 2$. 

Then we have $f=v\circ u$. 

Now we consider the preprojective algebra $\wLa$ of type $\tilde{\type{D}}_{5}$ given by the following quiver 

$$\xymatrix@C50pt@R20pt{&&&-1\ar@<0.5ex>[d]^{b_{-1}}&\\
&4 \ar[r]^{a_{3}}&\ar@<1ex>[l]^{b_{3}}3\ar@<0.5ex>[d]^{b_{-3}}\ar[r]^{a_{2}}&2 \ar@<1ex>[l]^{b_{2}}\ar[r]^{a_{1}}\ar@<0.5ex>[u]^{a_{-1}}& \ar@<1ex>[l]^{b_{1}}1.\\
&&0\ar@<0.5ex>[u]^{a_{-3}}&&} $$ 

Then we will give a lifting of $u,v\in\Aut(\La)$ to $\Aut(\wLa)$. First we consider $v$ and show the following lemma.

\begin{lemma}\label{lift v}
For $v$, we have $\tilde{v}\in\Aut(\wLa)$ such that
$$\xymatrix@C45pt@R15pt{
\wLa\ar[d]_{\rm nat.} \ar[r]^{\tilde{v}}   &\ar[d]^{\rm nat.}\wLa  \\
\La \ar[r]^{v}  &\La.}$$
\end{lemma}

\begin{proof}
If $c_3=0$ and $d_3=0$, then it is clear. 
Consider the case $c_3\neq0$ and $d_3\neq0$. 
Let $x:=(-c_3/d_3 )^{1/2}$. 
Define $g\in\Aut(\La)$ by 
$$g(a_1)=(1/x)a_1\ \textnormal{and}\ g(b_{1})=x b_{1}$$ 
 and fix all the idempotents and all the other arrows. 
Then we can clearly 
give a lifting $\widetilde{g}\in\Aut(\wLa)$ of $g$. 
Moreover, for $v_g:=g^{-1}\circ v\circ g$, we have 
$v_g(e_1)=e_1+c(b_{1}[3])+d(b_{-1}[3])$ and $v_g(e_{-1})=e_{-1}-c(b_{1}[3])-d(b_{-1}[3])$, where
$c:=c_3/x$ and $d:=d_3x=-c$.
Then we have $c+d=0$ and 
we can give a lifting $\widetilde{v_g}\in\Aut(\wLa)$ of $v_g$
by the similar argument of Proposition \ref{prop D4}. 
The case of $c_3=0$ and $d_3\neq 0$, or $c_3\neq0$ and $d_3=0$  follows from the second case.
\end{proof}

Next we consider $u$ and we give a correspondence 
$\tilde{u}(e_{\pm1}),\tilde{u}(a_{\pm1})$ and $\tilde{u}(b_{\pm1})$ by {\bf (i),(ii),(iii)}.

Then, by Lemma \ref{zero relation D_4}, we obtain the following result.

\begin{lemma}\label{zero relation D_5}
We have $\tilde{u}(b_{\pm 1})\tilde{u}(a_{\pm 1})=0.$
\end{lemma}

\begin{proof} 
By the condition (\ref{D_5}), the result follows from the same calculation of Lemma \ref{zero relation D_4}.
\end{proof}

Then we can obtain the following result.

\begin{lemma}\label{lift u}
For $u$, we have $\tilde{u}\in\Aut(\wLa)$ such that
$$\xymatrix@C45pt@R15pt{
\wLa\ar[d]_{\rm nat.} \ar[r]^{\tilde{u}}   &\ar[d]^{\rm nat.}\wLa  \\
\La \ar[r]^{u}  &\La}$$
by defining $\tilde{u}$ as follows $:$

$\bullet$ $\tilde{u}(e_{\pm1}),\tilde{u}(a_{\pm1})$ and $\tilde{u}(b_{\pm1})$ as {\bf (i),(ii),(iii)}.

$\bullet$ $\tilde{u}(e_2)=e_2$, $\tilde{u}(e_3)=e_3$, 
$\tilde{u}(a_2)=a_2$, $\tilde{u}(b_2)=k_1b_2$. 
\begin{eqnarray*}
\bullet\ \tilde{u}(e_4):=e_4+\tilde{f}_{\infty}(a_3,a_{-3}),\ \ \ 
\tilde{u}(e_{0}):=e_{0}-\tilde{f}_{\infty}(a_3,a_{-3}),\\
\tilde{u}(b_3):=k_1(b_3+b_{-3})\tilde{u}(e_4),\ \ \ 
\tilde{u}(b_{-3}):=k_1(b_3+b_{-3})\tilde{u}(e_{0}),\\
\tilde{u}(a_3):=\tilde{u}(e_4)(a_3+a_{-3}),\ \ \ 
\tilde{u}(a_{-3}):=\tilde{u}(e_{0})(a_3+a_{-3}).
\end{eqnarray*}
\end{lemma}

\begin{proof}
We need to check the following relations

\[\left\{\begin{array}{lll}
\widetilde{u}(b_{\pm 1})\widetilde{u}(a_{\pm 1})=0 &&\textnormal{(a)}\\
 \widetilde{u}(a_{\pm 3})\widetilde{u}(b_{\pm 3})=0 
&&\textnormal{(b)}\\
\widetilde{u}(b_2) \widetilde{u}(a_2)=\widetilde{u}(a_1+a_{-1})\widetilde{u}(b_1+b_{-1}) &&\textnormal{(c)}\\
\widetilde{u}(a_{2})\widetilde{u}(b_{2})=\widetilde{u}(b_{3}+b_{-3})\widetilde{u}(a_{3}+a_{-3}) &&\textnormal{(d)}
 \end{array}\right.\]

Note that we have $\tilde{u}(e_{4})+\tilde{u}(e_{0})=e_4+e_0$.

(a) This follows from Lemma \ref{zero relation D_5} and it also implies (b) similarly.

(c) This follows from 
$\tilde{u}(b_{2})\tilde{u}(a_{2})=k_1b_{2}a_{2}=
(a_1+a_{-1})k_1(b_1+b_{-1})=\tilde{u}(a_1+a_{-1})\tilde{u}(b_1+b_{-1})$ and (d) follows from the same argument.

Thus $\widetilde{u}$ gives a morphism of $\Aut(\wLa)$. 
The second statement is clear from the definition of $u$ and $\widetilde{u}$.
\end{proof}

As a consequence of Lemma \ref{lift v} and \ref{lift u}, we obtain the desired conclusion for $\type{D}_5$.

\subsection{The case $\type{D}_{n+1}$ for even.}

In this subsection, we show Proposition \ref{restriction} in the case of $\type{D}_{n+1}$ when $n+1$ is even. This is shown by the same argument of $\type{D}_{4}$.

Assume that $n+1$ ($n\geq3$) is even and consider $\type{D}_{n+1}$.
Fix $f\in H$. 
By Lemma \ref{P_1}, without loss of generality, 
we can write 
\begin{eqnarray*}
f(e_1)=e_1+\sum^{n-1}_{j=1}c_j(b_1 [j])+\sum^{n-1}_{j=1}d_j(b_{-1}[j]), \\
f(e_{-1})=e_{-1}-\sum^{n-1}_{j=1}c_j(b_{1}[j])-\sum^{n-1}_{j=1}d_j(b_{-1}[j]),
\end{eqnarray*}
where  
$$-c_j=c_1d_{j-1}+c_2c_{j-2}+c_3d_{j-3}+c_4c_{j-4}+\cdots+c_{j-2}c_{2}+c_{j-1}d_{1}=d_{j}$$ 
for even $j$.

Then, we have
$f(a_1)=(a_1+a_{-1})f(e_1)=a_1+\sum^{n-2}_{j=1}c_ja_1(b_{1}[j])+\sum^{n-2}_{j=1}d_ja_{-1}(b_{-1}[j])$
and $f(a_{-1})=a_{-1}-\sum^{n-2}_{j=1}c_ja_1(b_{1}[j])-\sum^{n-2}_{j=1}d_ja_{-1}(b_{-1}[j]).$

Furthermore, since $f(b_1+b_{-1})=\sum^{n-2}_{j=1}k_j(b_1+b_{-1})(b_2a_2)^{j-1}
=\sum^{n-2}_{j=1}k_j\{( b_1[j-1])b_{\overline{1}}+(b_{-1}[j-1])b_{\overline{-1}})\}$, we have
\begin{eqnarray*}
f(b_1)&=&k_1\{b_1+\sum^{n-2}_{j=1}c_j(b_{1}[j])b_{\overline{1}}+\sum^{n-2}_{j=1}d_j(b_{-1}[j])b_{\overline{-1}}\}\\
&+&k_2\{(b_1[1])b_{-1}+\sum^{n-3}_{j=1}c_j(b_{1}[j+1])b_{\overline{1}}+\sum^{n-3}_{j=1}d_j(b_{-1}[j+1])b_{\overline{-1}}\}\\
&\vdots&\cdots\\
&+&k_{n-2}\{ (b_1 [n-3])b_{\overline{1}}+\sum^{1}_{j=1}c_j(b_{1}[j+n-3])b_{\overline{1}}+\sum^{1}_{j=1}d_j(b_{-1}[j+n-3])b_{\overline{-1}}\},\\
f(b_{-1})&=&k_1\{b_{-1}-\sum^{n-2}_{j=1}c_j(b_{1}[j])b_{\overline{1}}-\sum^{n-2}_{j=1}d_j(b_{-1}[j])b_{\overline{-1}}\}\\
&+&k_2\{(b_{-1}[1])b_{1}-\sum^{n-3}_{j=1}c_j(b_{1}[j+1])b_{\overline{1}}-\sum^{n-3}_{j=1}d_j(b_{-1}[j+1])b_{\overline{-1}}\}\\
&\vdots&\cdots\\
&+&k_{n-2}\{ (b_{-1} [n-3])b_{\overline{1}}-\sum^{1}_{j=1}c_j(b_{1}[j+n-3])b_{\overline{1}}-\sum^{1}_{j=1}d_j(b_{-1}[j+n-3])b_{\overline{-1}}\}.
\end{eqnarray*}

By the relations $f(b_1)f(a_1)=0$ and $f(b_{-1})f(a_{-1})=0$, we have the following conditions.

\begin{lemma}\label{vanish f even}
For any even $m$ with $2\leq m \leq n-2$, we have $k_m=0$ and $c_{m-1}+d_{m-1}=0.$ 
\end{lemma}

\begin{proof} 
Because $f\in \Aut(\La)$, we have $f(b_1)f(a_1)=0$ and $f(b_{-1})f(a_{-1})=0$, and we show that these conditions implies the desired result.

(i) First we calculate $f(b_1)f(a_1)$. 
We denote by $\Co(b_{\pm 1} [j])$ the coefficient of $(b_{\pm 1} [j])$ in $f(b_1)f(a_1)$  for any $j$.  
Recall that we have $(b_{\pm 1} [j])=0$ for any $n-1\leq j$
 and hence 
we will check $\Co(b_{\pm 1} [j])$ for  $j\leq n-2$.

Fix even $m$ with $m\leq n-2$. By the direct calculation, we can check 

\begin{eqnarray*}
\Co(b_1[m])&=&k_1\{(c_{m-1}+d_{m-1})+c_2(c_{m-3}+d_{m-3})+c_4(c_{m-5}+d_{m-5})+\cdots+c_{m-2}(c_{1}+d_{1})   \} \nonumber \\
&+&k_2c_{m-2}\nonumber \\
&+&k_3\{(c_{m-3}+d_{m-3})+c_2(c_{m-5}+d_{m-5})+\cdots+c_{m-4}(c_{1}+d_{1})   \}\nonumber \\
&+&k_4c_{m-4}\\
&\vdots&\cdots\nonumber \\
&+&k_{m-1}(c_1+d_1)\nonumber \\
&+&k_{m}\nonumber, \\
\Co(b_1 [m+1])&=&k_1\{c_1(c_{m-1}+d_{m-1})+c_3(c_{m-3}+d_{m-3})+c_5(c_{m-5}+d_{m-5})+\cdots+c_{m-1}(c_{1}+d_{1})   \}\\
&+&k_2c_{m-1}\\
&+&k_3\{c_1(c_{m-3}+d_{m-3})+c_3(c_{m-5}+d_{m-5})+\cdots+c_{m-3}(c_{1}+d_{1})   \}\\
&+&k_4c_{m-3}\\
&\vdots&\cdots\\
&+&k_{m-1}\{c_1(c_1+d_1)\}\\
&+&k_{m}c_1,\\
\Co(b_{-1}[m])&=&k_1\{d_2(c_{m-3}+d_{m-3})+d_4(c_{m-5}+d_{m-5})+d_6(c_{m-7}+d_{m-7})+\cdots+d_{m-2}(c_{1}+d_{1})   \} \nonumber \\
&+&k_2d_{m-2}\nonumber \\
&+&k_3\{(d_2(c_{m-5}+d_{m-5})+d_4(c_{m-7}+d_{m-7})+\cdots+d_{m-4}(c_{1}+d_{1})   \}\nonumber \\
&+&k_4d_{m-4}\\
&\vdots&\cdots\nonumber \\
&+&k_{m-3}\{d_2(c_1+d_1)\}\nonumber \\
&+&k_{m-2}d_2\nonumber, \\
\Co(b_{-1} [m+1])&=&k_1\{d_1(c_{m-1}+d_{m-1})+d_3(c_{m-3}+d_{m-3})+d_5(c_{m-5}+d_{m-5})+\cdots+d_{m-1}(c_{1}+d_{1})   \}\\
&+&k_2d_{m-1}\\
&+&k_3\{d_1(c_{m-3}+d_{m-3})+d_3(c_{m-5}+d_{m-5})+\cdots+d_{m-3}(c_{1}+d_{1})   \}\\
&+&k_4d_{m-3}\\
&\vdots&\cdots\\
&+&k_{m-1}\{d_1(c_1+d_1)\}\\
&+&k_{m}d_1.
\end{eqnarray*}

(ii) Next, we calculate the coefficients of $f(b_{-1})f(a_{-1})$. 
Similarly we have 

\begin{eqnarray*}
\Co(b_1[m])&=&k_1\{c_2(c_{m-3}+d_{m-3})+c_4(c_{m-5}+d_{m-5})+c_6(c_{m-7}+d_{m-7})+\cdots+c_{m-2}(c_{1}+d_{1})   \} \nonumber \\
&+&k_2d_{m-2}\nonumber \\
&+&k_3\{c_2(c_{m-5}+d_{m-5})+c_4(c_{m-7}+d_{m-7})+\cdots+c_{m-4}(c_{1}+d_{1})   \}\nonumber \\
&+&k_4d_{m-4}\\
&\vdots&\cdots\nonumber \\
&+&k_{m-3}\{c_2(c_1+d_1)\}\nonumber \\
&+&k_{m-2}d_2\nonumber, \\
\Co(b_1 [m+1])&=&k_1\{c_1(c_{m-1}+d_{m-1})+c_3(c_{m-3}+d_{m-3})+c_5(c_{m-5}+d_{m-5})+\cdots+c_{m-1}(c_{1}+d_{1})   \}\\
&-&k_2c_{m-1}\\
&+&k_3\{c_1(c_{m-3}+d_{m-3})+c_3(c_{m-5}+d_{m-5})+\cdots+c_{m-3}(c_{1}+d_{1})   \}\\
&-&k_4c_{m-3}\\
&\vdots&\cdots\\
&+&k_{m-1}\{c_1(c_1+d_1)\}\\
&-&k_{m}c_1,\\
\Co(b_{-1}[m])&=&k_1\{(c_{m-1}+d_{m-1})+c_2(c_{m-3}+d_{m-3})+c_4(c_{m-5}+d_{m-5})+\cdots+c_{m-2}(c_{1}+d_{1})   \} \nonumber \\
&+&k_2d_{m-2}\nonumber \\
&+&k_3\{(c_{m-3}+d_{m-3})+c_2(c_{m-5}+d_{m-5})+\cdots+c_{m-4}(c_{1}+d_{1})   \}\nonumber \\
&+&k_4d_{m-4}\\
&\vdots&\cdots\nonumber \\
&+&k_{m-1}(c_1+d_1)\nonumber \\
&-&k_{m}\nonumber, \\
\Co(b_{-1} [m+1])&=&k_1\{d_1(c_{m-1}+d_{m-1})+d_3(c_{m-3}+d_{m-3})+d_5(c_{m-5}+d_{m-5})+\cdots+d_{m-1}(c_{1}+d_{1})   \}\\
&-&k_2d_{m-1}\\
&+&k_3\{d_1(c_{m-3}+d_{m-3})+d_3(c_{m-5}+d_{m-5})+\cdots+d_{m-3}(c_{1}+d_{1})   \}\\
&-&k_4d_{m-3}\\
&\vdots&\cdots\\
&+&k_{m-1}\{d_1(c_1+d_1)\}\\
&-&k_{m}d_1.
\end{eqnarray*}

Since ${f}(b_{1}){f}(a_{1})=0$ and ${f}(b_{-1}){f}(a_{-1})=0$, 
all coefficients are zero. 

First, consider the case of $m=2$. Then from (i) and (ii), we have 
$k_{1}(c_1+d_1)+k_2=0$ and $k_{1}(c_1+d_1)-k_2=0$.
Therefore we have $k_2=0$ and $c_1+d_1=0$. 

Next consider the case of $m=4$. Then the above calculations  similarly imply that $k_4=0$ and $c_3+d_3=0$. 
Inductively, we can get $k_m=0$ and $c_{m-1}+d_{m-1}=0$ for any even $m$ with $2\leq m\leq n-2$, and we get the conclusion. 
\end{proof}

Next using the coefficients $c_j,d_j$ ($j\leq n-1$), we give the following correspondence, and we will show that it gives an automorphism of $\wLa$.

$\bf{(i)}'$ First we define
\begin{eqnarray*}
\tilde{f}(e_1):=e_1+\tilde{f}_{\infty}(b_1,b_{-1}),\ \ \ 
\tilde{f}(e_{-1}):=e_1-\tilde{f}_{\infty}(b_1,b_{-1}), 
\end{eqnarray*}
where we define $c_j$ and $d_j$ ($j\geq n$) as follows.
For odd $j$, we put $c_{j}=0=d_j$ for any $j$. 
For even $j$, we define 
$$-c_j=c_1d_{j-1}+c_2c_{j-2}+c_3d_{j-3}+c_4c_{j-4}+\cdots+c_{j-2}c_{2}+c_{j-1}d_{1}=d_{j}.$$

Note that 
we have $\tilde{f}(e_1)+\tilde{f}(e_{-1})=e_1+e_{-1}$, $\tilde{f}(e_1)^2=\tilde{f}(e_1)$, $\tilde{f}(e_{-1})^2=\tilde{f}(e_{-1})$, 
$\tilde{f}(e_1)\tilde{f}(e_{-1})=0$ and $\tilde{f}(e_{-1})\tilde{f}(e_1)=0$.

$\bf{(ii)}'$ Secondly, we define  
$\tilde{f}(a_1):=(a_1+a_{-1})\tilde{f}(e_1)$ and $\tilde{f}(a_{-1}):=(a_1+a_{-1})\tilde{f}(e_{-1})$.

$\bf{(iii)}'$ Thirdly,  
we define $\tilde{f}(b_1):=\tilde{f}(e_1)\sum^{n-2}_{j=1}k_j(b_1+b_{-1})(b_2a_2)^{j-1}$ and  
$\tilde{f}(b_{-1}):=\tilde{f}(e_{-1})\sum^{n-2}_{j=1}k_j(b_1+b_{-1})(b_2a_2)^{j-1}$. 
Note that, since $\sum^{n-2}_{j=1}k_j(b_1+b_{-1})(b_2a_2)^{j-1}=\sum^{n-2}_{j=1}k_j\{ (b_1[j-1])b_{\overline{1}}+(b_{-1}[j-1])b_{\overline{-1}}\}$, 
we can write  
\begin{eqnarray*}
\tilde{f}(b_1)
&=&k_1\{b_1+\sum^{\infty}_{j=1}c_j(b_{1}[j])b_{\overline{1}}+\sum^{\infty}_{j=1}d_j(b_{-1}[j])b_{\overline{-1}}\}\\
&+&k_2\{(b_1[1])b_{-1}+\sum^{\infty}_{j=1}c_j(b_{1}[j+1])b_{\overline{1}}+\sum^{\infty}_{j=1}d_j(b_{-1}[j+1])b_{\overline{-1}}\}\\
&\vdots&\cdots\\
&+&k_{n-2}\{ (b_1 [n-3])b_{1}+\sum^{\infty}_{j=1}c_j(b_{1}[j+n-3])b_{\overline{1}}+\sum^{\infty}_{j=1}d_j(b_{-1}[j+n-3])b_{\overline{-1}}\}.
\end{eqnarray*}

In this setting, we can check the following lemma.

\begin{lemma}\label{zero relation even}
We have $\tilde{f}(b_{\pm 1})\tilde{f}(a_{\pm 1})=0.$
\end{lemma}

\begin{proof}
We only show $\tilde{f}(b_{1})\tilde{f}(a_{1})=0$, and $\tilde{f}(b_{-1})\tilde{f}(a_{-1})=0$ can be shown by the same argument. 
In the case of $j\leq n-1$, 
the coefficient $\Co(b_{\pm 1}[j])$ in $\tilde{f}(b_{ 1})\tilde{f}(a_{1})$ is the same as  
the coefficient $\Co(b_{\pm 1}[j])$ in ${f}(b_{ 1}){f}(a_{1})$, and hence they are zero. 

Moreover the same calculation of Lemma \ref{vanish f even}
 implies that $\Co(b_{\pm 1}[j])=0$ in $\tilde{f}(b_{1})\tilde{f}(a_{1})$ for any $n\leq j$ because $c_j=0=d_j$ for odd $j\geq n$ and $k_m=0$ for even $m\geq 2$.
\end{proof}

Then we give a proof of Proposition \ref{restriction} as follows.

\begin{proposition}
For $f\in H$, we have $\tilde{f}\in\Aut(\wLa)$ such that
$$\xymatrix@C45pt@R15pt{
\wLa\ar[d]_{\rm nat.} \ar[r]^{\tilde{f}}   &\ar[d]^{\rm nat.}\wLa  \\
\La \ar[r]^{f}  &\La}$$
by defining $\tilde{f}$ as follows.

$\bullet$ $\tilde{f}(e_{\pm1}),\tilde{f}(a_{\pm1})$ and $\tilde{f}(b_{\pm1})$ as ${\bf (i)',(ii)',(iii)'}$.

$\bullet$ 
$\tilde{f}(e_{i}):=e_i$  $(2\leq i\leq n-1)$, $\tilde{f}(a_{i}):=a_i$ and $\tilde{f}(b_{i}):=\sum^{n-2}_{j=1}k_jb_i(a_ib_i)^{j-1}$ $(2\leq i\leq n-2)$.

\begin{eqnarray*}
\bullet\ \tilde{f}(e_n):=e_n+\tilde{f}_{\infty}(a_{n-1},a_{-(n-1)}),\ \ \ 
\tilde{f}(e_{0}):=e_{0}-\tilde{f}_{\infty}(a_{n-1},a_{-(n-1)}),\\
\tilde{f}(b_{n-1}):=k_1(b_{n-1}+b_{-(n-1)})\tilde{f}(e_n),\ \ \ 
\tilde{f}(b_{-(n-1)}):=k_1(b_{n-1}+b_{-(n-1)})\tilde{f}(e_{0}),\\
\tilde{f}(a_{n-1}):=\tilde{f}(e_n)\sum^{n-2}_{j=1}k_j/k_1(a_{n-1}+a_{-(n-1)})(a_{n-2}b_{n-2})^{j-1},\\
\tilde{f}(a_{-(n-1)}):=\tilde{f}(e_{0})\sum^{n-2}_{j=1}k_j/k_1(a_{n-1}+a_{-(n-1)})(a_{n-2}b_{n-2})^{j-1}.
\end{eqnarray*}
\end{proposition}

\begin{proof}
We will check the following relations

\[\left\{\begin{array}{lll}
\widetilde{f}(b_{\pm 1})\widetilde{f}(a_{\pm 1})=0 &&\textnormal{(a)}\\
 \widetilde{f}(a_{\pm(n-1)})\widetilde{f}(b_{\pm(n-1)})=0 
&&\textnormal{(b)}\\
\widetilde{f}(b_2) \widetilde{f}(a_2)=\widetilde{f}(a_1+a_{-1})\widetilde{f}(b_1+b_{-1}) &&\textnormal{(c)}\\
\widetilde{f}(a_{i-1})\widetilde{f}(b_{i-1})=\widetilde{f}(b_{i}) \widetilde{f}(a_{i})\ \ (3\leq i\leq n-2 ) &&\textnormal{(d)}\\
\widetilde{f}(a_{n-2})\widetilde{f}(b_{n-2}) =\widetilde{f}(b_{n-1}+b_{-(n-1)})\widetilde{f}(a_{n-1}+a_{-(n-1)}) &&\textnormal{(e)}
 \end{array}\right.\]

(a) This follows from Lemma \ref{zero relation even} and it also implies (b).

(c) We have 
\begin{eqnarray*}
\widetilde{f}(b_2) \widetilde{f}(a_2)&=
& (\sum^{n-2}_{j=1}k_jb_2(a_2b_2)^{j-1})a_2\\
&=& \sum^{n-2}_{j=1}k_j(b_2a_2)^{j}\\
&=& (a_1+a_{-1})(\sum^{n-2}_{j=1}k_j(b_1+b_{-1})(b_2a_2)^{j-1})\\
&=& \widetilde{f}(a_1+a_{-1})\widetilde{f}(b_1+b_{-1})
\end{eqnarray*}

and (d), (e) are similar.

Thus $\widetilde{f}$ gives a morphism of $\Aut(\wLa)$. 
The second statement is clear from the definition from $f$ and $\widetilde{f}$.
\end{proof}

\subsection{The case $\type{D}_{n+1}$ for odd.}

Finally we deal with the case $\type{D}_{n+1}$ when $n+1$ is odd, and complete the proof. This is shown by the same argument of $\type{D}_5$.

Assume that $n+1$ is odd and consider $\type{D}_{n+1}$ ($n\geq3$). 
Fix $f\in H$. 
Then we can write 
$f(e_1)=e_1+\sum^{n-1}_{j=1}c_j(b_1 [j])+\sum^{n-1}_{j=1}d_j(b_{-1}[j])$ and 
$f(e_{-1})=e_{-1}-\sum^{n-1}_{j=1}c_j(b_{1}[j])-\sum^{n-1}_{j=1}d_j(b_{-1}[j]),$ 
where  
$$-c_j=c_1d_{j-1}+c_2c_{j-2}+c_3d_{j-3}+c_4c_{j-4}+\cdots+c_{j-2}c_{2}+c_{j-1}d_{1}=d_{j}$$ 
for even $j$.

We consider the factorization of $f$ using $u,v\in\Aut(\La)$, which is defined as follows.  
First $u$ is defined by $u(e_1):=e_{1}+\sum^{n-2}_{j=1}c_j(b_{1}[j])+\sum^{n-2}_{j=1}d_j(b_{-1}[j])$ and $u(e_{-1}):=e_{-1}-\sum^{n-2}_{j=1}c_j(b_{1}[j])-\sum^{n-2}_{j=1}d_j(b_{-1}[j]),$ and
$u(e_i):=f(e_i)$, $u(a_i):=f(a_i)$ and $u(b_i):=f(b_i)$ for all $i\geq 2$ (i.e. $f=u$ if $c_{n-1}=0=d_{n-1}$).

Second $v$ is defined by $v(e_1):=e_1+c_{n-1}(b_{1}[n-1])+d_{n-1}(b_{-1}[n-1])$ and  $v(e_{-1}):=e_{-1}-c_{n-1}(b_{1}[n-1])-d_{n-1}(b_{-1}[n-1])$, and 
$v(e_i):=e_i$, $v(a_i):=a_i$ and $v(b_i):=b_i$ for $i\geq 2$. 
Then we can check that 
$f=v\circ u$. 

The following lemmas follow from the same argument of 
Lemma \ref{lift v} and \ref{lift u}.

\begin{lemma}\label{lift v n+1}
For $v$, we have $\tilde{v}\in\Aut(\wLa)$ such that
$$\xymatrix@C45pt@R15pt{
\wLa\ar[d]_{\rm nat.} \ar[r]^{\tilde{v}}   &\ar[d]^{\rm nat.}\wLa  \\
\La \ar[r]^{v}  &\La.}$$
\end{lemma}

\begin{lemma}\label{lift u n+1}
For $u$, we have $\tilde{u}\in\Aut(\wLa)$ such that
$$\xymatrix@C45pt@R15pt{
\wLa\ar[d]_{\rm nat.} \ar[r]^{\tilde{u}}   &\ar[d]^{\rm nat.}\wLa  \\
\La \ar[r]^{u}  &\La}$$
by defining $\tilde{u}$ as follows $:$

$\bullet$ $\tilde{u}(e_{\pm1}),\tilde{u}(a_{\pm1})$ and $\tilde{u}(b_{\pm1})$ as {\bf (i)$',$(ii)$'$,(iii)$'$}.

$\bullet$ 
$\tilde{u}(e_{i})=e_i$  $(2\leq i\leq n-1)$, $\tilde{u}(a_{i})=a_i$ and $\tilde{u}(b_{i})=\sum^{n-3}_{j=1}k_jb_i(a_ib_i)^{j-1}$ $(2\leq i\leq n-2)$.

\begin{eqnarray*}
\bullet\ \tilde{u}(e_n):=e_n+\tilde{f}_{\infty}(a_{n-1},a_{-(n-1)}),\ \ \ 
\tilde{u}(e_{0}):=e_{0}-\tilde{f}_{\infty}(a_{n-1},a_{-(n-1)}),\\
\tilde{u}(b_{n-1}):=k_1(b_{n-1}+b_{-(n-1)})\tilde{u}(e_n),\ \ \ 
\tilde{u}(b_{-(n-1)}):=k_1(b_{n-1}+b_{-(n-1)})\tilde{u}(e_{0}),\\
\tilde{u}(a_{n-1}):=\tilde{u}(e_n)\sum^{n-3}_{j=1}k_j/k_1(a_{n-1}+a_{-(n-1)})(a_{n-2}b_{n-2})^{j-1},\\ 
\tilde{u}(a_{-(n-1)}):=\tilde{u}(e_{0})\sum^{n-3}_{j=1}k_j/k_1(a_{n-1}+a_{-(n-1)})(a_{n-2}b_{n-2})^{j-1}.
\end{eqnarray*}
\end{lemma}

As a consequence of Lemma \ref{lift v n+1} and \ref{lift u n+1}, we obtain the conclusion in $\type{D}_{n+1}$ for odd $n+1$.\\


\textbf{Acknowledgements.} 

The author is grateful to Osamu Iyama and Ryoichi Kase for their kind advice and valuable suggestions. 
He also thanks Yuta Kozakai for many comments  on the manuscript, which improve the presentation, and Kota Yamaura for stimulating discussions.

The author was supported by Grant-in-Aid for JSPS Research Fellow 17J00652.

\end{document}